\documentclass[12pt]{amsart}

\usepackage[usenames,dvipsnames]{color}
\usepackage{array}
\usepackage{indentfirst,pict2e}
\usepackage{amsfonts,amssymb}
\usepackage[leqno]{amsmath}
\usepackage{amsthm}
\usepackage{latexsym}
\usepackage[margin=1.0in]{geometry}
\usepackage{youngtab, ytableau}

	\definecolor{linkred}{rgb}{0.7,0.2,0.2}
	\definecolor{linkblue}{rgb}{0,0.2,0.6}

\newcommand\numberthis{\addtocounter{equation}{1}\tag{\theequation}}

\usepackage[backrefs,msc-links,nobysame]{amsrefs}

\makeatletter

\def\bib@div@mark#1{%
 \@mkboth{{#1}}{{#1}}%
	}
\def\print@backrefs#1{%
 \space\SentenceSpace$\leftarrow$\csname br@#1\endcsname
}
\catcode`\'=11 
\renewcommand{\PrintAuthors}[1]{%
 \ifx\previous@primary\current@primary
  \sameauthors\@empty
 \else
  \def\current@bibfield{\bib'author}%
		  \PrintNames{}{}{\scshape #1}%
 \fi
}
\catcode`\'=12
\def\MRhref#1#2{%
 \begingroup
  \parse@MR#1 ()\@empty\@nil%
  \href{\MR@url}{\texttt{\@tempd\vphantom{()}}}%
  \ifx\@tempe\@empty
  \else
   \ \href{\MR@url}{\texttt{(\@tempe)}}%
  \fi
 \endgroup
}%
\def\MR#1{%
 \relax\ifhmode\unskip\spacefactor3000 \space\fi
 \begingroup
  \strip@MRprefix#1\@nil
  \edef\@tempa{\@nx\MRhref{MR\@tempa}{\@tempa}}%
 \@xp\endgroup
 \@tempa
}
\makeatother

\numberwithin{equation}{section}

\newtheorem{theorem*}{Theorem}
\newtheorem{problem*}{Problem}
\newtheorem{definition}[equation]{Definition}
\newtheorem{claim}[equation]{Claim}
\newtheorem{lemma}[equation]{Lemma}
\newtheorem{corollary}[equation]{Corollary}

\newtheorem{proposition}[equation]{Proposition}

\newtheorem{remark}[equation]{Remark}
\newtheorem{fact}[equation]{Fact}
\newtheorem{example}{Example}
\newtheorem{observation}[equation]{Observation}

\newcommand{\M}{\overline{\operatorname{M}}}

\newcommand{\rkk}{\operatorname{rk}}
\newcommand{\sL}{\mathfrak{sl}}
\newcommand{\sP}{\mathfrak{sp}}

\newcommand{\rk}{\operatorname{rank}}
\newcommand{\VsP}{\mathbb{V}_{\sP_{2\ell}, \vec{\lambda}, 1}}
\newcommand{\VsL}{\mathbb{V}_{\sL_{2}, \vec{\lambda}, \ell }}

\newcommand{\VsLi}{\mathbb{V}_{\sL_{2}, \vec{\lambda}_i, 1 }}
\newcommand{\VsLl}{\mathbb{V}_{\sL_{2}, \vec{\lambda}_{\ell}, \ell }}
\newcommand{\VsLj}{\mathbb{V}_{\sL_{2}, \vec{\lambda}_j, 1 }}
\newcommand{\st}{\tilde{s}}
\newcommand{\m}{\mu}
\newcommand{\mt}{\tilde{\mu}}

\newcommand{\VlL}{\mathbb{V}_{\sP_{2\ell}, \vec{\lambda}, 1}}

\newcommand{\rlam}{r_{\vec{\lambda}}} 
 
\newcommand{\Vrlam}{\mathbb{V}_{\sP_{2\rlam}, \vec{\lambda}, 1}}

\begin{document}

\pagenumbering{arabic}

\title{Conformal Blocks in type C at level one}
\author{Natalie Hobson}
\date{\today}
\maketitle


\begin{abstract}
We investigate the behavior of  vector bundles of conformal blocks for $\sP_{2\ell}$ at level one on $\M_{0,n}$. We show their first Chern classes are equivalent to conformal blocks divisors for $\sL_2$ at level $\ell$ if and only if the corresponding vector bundles have rank one or zero, and all become equal when the Lie algebra $\ell$ is large enough. As a consequence of these results, we conclude the cone generated by these divisors is polyhedral. 
\end{abstract}

\section{Introduction}

Conformal blocks divisors are first Chern classes of certain vector bundles defined on the stack $\overline{\mathcal{M}}_{g,n}$. These bundles $\mathbb{V}(\mathfrak{g}, \vec{\lambda}, \ell)$ are determined by a simple Lie algebra $\mathfrak{g}$, a positive integer $\ell$, and an $n$-tuple $\vec{\lambda}$ of dominant integral weight for $\mathfrak{g}$ at level $\ell$. In genus zero, the vector bundles are globally generated, and so define base point free, and hence nef divisors on $\overline{\operatorname{M}}_{0,n}$, the smooth projective variety which represents $\overline{\mathcal{M}}_{0,n}.$ In this work, we focus our attention to this case.  

One recurring theme is that these divisor classes satisfy identities of various types \cite{ags, agss, BGMA,BGMB, Fakhruddin, GG, Hobson, Kaz, MUK13, MUK14, swinarskisl2}. There is interest in such relations, as there is some question of just how many distinct conformal blocks divisors span extremal rays of the nef cone. 

 In this work we:

\begin{enumerate}
\item[1. ] show in rank one, $c_1(\mathbb{V}(\sP_{2\ell}, \vec{\lambda}, 1))$ are equal to those of type A at level $\ell$ (Prop. \ref{main}),
\item[2. ] prove for $r$ large enough, $c_1(\mathbb{V}(\sP_{2r}, \vec{\lambda}, 1))$ are all equal (Prop. \ref{rsdivisor}), and
\item[3. ] describe rank behavior for $\mathbb{V}(\sP_{2\ell}, \vec{\lambda}, \ell)$ as the Lie algebra rank varies (Prop. \ref{rsrank}). 
\end{enumerate}

Even at level one, there are an infinite number of conformal blocks divisors, and so we can ask whether these divisors span a polyhedral cone. In types $A$ and $D$, level one bundles have rank one. It was shown by Giansiracusa and Gibney in \cite[Thm.~1.1]{GG} and Fakhruddin in \cite[Prop.~5.6]{Fakhruddin} that the cone of level one divisors in types A and D (respectively) is finitely generated. Our main result is an identity between divisors with different Lie data.

\begin{proposition}\label{main}
Define $\VsL$ and $\VsP$ as in Def. \ref{def1}, the first Chern class equality 
$$c_1(\VsL) = c_1(\VsP)$$
holds if and only if $\rk(\VsL) = 1$ or $0.$
\end{proposition}

As a consequence of our first result, we show finite generation for divisors of bundles of type C at level one and rank one.

\begin{corollary}
The cone $$\mathcal{C}=ConvHull \{ c_1(\VsP) : \rk(\VsP) = 1 \},$$ is a full dimensional, finitely generated subcone of the nef cone of $\overline{\operatorname{M}}_{0,n}$.
\end{corollary}

In \cite[Thm.~1.1]{Hobson} we explicitly decompose $c_1(\VsP)$ for $\rkk(\VsP) = 1$ into an effective sum of level one divisors of $\sL_2$ bundles. Using this and results from \cite{GG} we describe the maps given by $c_1(\VsP)$ as a map to a product of GIT quotients known as \textit{generalized Veronese quotients} (see Claim \ref{Maps}).

Furthermore, we give an explicit value $\rlam$, such that for any $r \geq \rlam$ the first Chern classes $c_1(\mathbb{V}(\sP_{2r}, \vec{\lambda}, 1))$ become equal, and in particular are nontrivial if $\rkk(\mathbb{V}(\sP_{2\rlam}, \vec{\lambda}, 1)) > 0.$ 


\begin{proposition}\label{rsdivisor}
Let $\rlam$ be the \textit{stabilizing Lie rank} (see Def. \ref{rlam}) associated to a fixed vector $\vec{\lambda} =(a_1, ..., a_n)$ of nonnegative integers. If $r \geq \rlam $ then
$$c_1(\mathbb{V}_{\sP_{2\rlam}, \vec{\lambda}, 1})= c_1( \mathbb{V}_{\sP_{2r}, \vec{\lambda}, 1}),$$ and particularly, these divisors are nontrivial if and only if $\rk(\mathbb{V}(\sP_{2\rlam}, \vec{\lambda}, 1)) > 0$.
\end{proposition}

We provide a relationship on the ranks of such bundles when the Lie algebra rank varies below this range. In the language of $\sL_2$ bundles, we show the vector bundle ranks strictly decrease when the level strictly decreases below the \textit{critical level} (see Remark~\ref{CL}). 


\begin{proposition}\label{rsrank}
For integers $r$ and $r'$, such that, $r > \rlam \geq r' $ and $\vec{\lambda} \in P_1(\sP_{2r'})^n$, we have
$$\rk(\mathbb{V}_{\sP_{2(r)}, \vec{\lambda}, 1}) = \rk(\mathbb{V}_{\sP_{2(\rlam +1)}, \vec{\lambda}, 1}) > \rk(\mathbb{V}_{\sP_{2(r')}, \vec{\lambda}, 1}).$$
\end{proposition}

\subsection{Outline of paper} In Section \ref{definitions} we establish our notation and state definitions and previous results. In Section \ref{proof} we prove Proposition \ref{main}. We describe next in Section \ref{Veronese}, the targets of the maps given by the divisors in Proposition~\ref{main}. In Section \ref{sl2ranks} we prove Proposition \ref{rsrank} in the context of ranks of $\sL_2$ bundles; this result is used in the proof of Proposition \ref{rsdivisor} in Section \ref{stabilizingproof}. We give examples in Section \ref{examples} showing the necessity of rank one in Proposition \ref{main}. We conclude in Section \ref{corollaries} with collection of many important open questions in the study of conformal blocks divisors and show several corollaries from our results that begin to solve these questions.  



\subsection{Acknowledgements}  The author would like to thank Angela Gibney and Dave Swinarski for many useful discussions and feedback. Particularly, she thanks Swinarski for his insights in Lemma~\ref{DegreeSPRank} and Claim~\ref{plussing}. She used Swinarski's Macaulay2 package \cite{DaveM2} for many computations in this work. The author acknowledges support from the RTG in Algebra, Algebraic Geometry, and Number Theory, at the University of Georgia (NSF grant DMS-1344994).

\section{Definitions and Previous Results} \label{definitions}

For the Lie algebra $\sL_2$ a nonzero dominant integral weight at level $\ell$ is given by a nonnegative integer multiple less than or equal to $\ell$ of the fundamental dominant weight $\omega_1$. For the Lie algebra $\sP_{2\ell}$ a nonzero dominant integral weight at level one is given by a single fundamental dominant weight $\omega_a$, where $a$ is an integer such that $0 \leq a \leq \ell$. The set of dominant integral weights at level $\ell$ for $\sL_2$ and at level one for $\sP_{2\ell}$ are thus given by the following sets respectively:
 $$P(\sL_2, \ell) = \{ i \omega_1\}_{i=0}^{\ell} \text{ and } P(\sP_{2\ell}, 1) = \{ \omega_{i} \}_{i=0}^{\ell}.$$

In this way, an $n$-tuple of integers $\vec{\lambda}=(a_1, ..., a_n)$ such that $0 < a_i \leq \ell$ for some integer $\ell$ specifies an $n$-tuple of weights from $P(\sL_2, \ell)$ or $P(\sP_{2\ell}, 1)$. It will often be convenient to fix such an $n$-tuple and define the vector bundles of conformal blocks associated to this weight vector for $\sL_2$ at level $\ell$ and $\sP_{2\ell}$ at level one. Throughout this paper, we always assume the integers in $\vec{\lambda}$ are weakly decreasing and the sum $|\vec{\lambda}| = \sum_{i=1}^n a_i$ is even. These vector bundles play the lead role in our story; we simplify notation by giving the following definition.  

\begin{definition} \label{def1}
Let $\vec{\lambda}=(a_1, ..., a_n)$ be a fixed $n$-tuple of weakly decreasing integers such that the sum $|\vec{\lambda}| = \sum_{i=1}^n a_i$ is even. Let $\ell$ be some fixed integer $\ell \geq a_1$, we define the following vector bundles $$\mathbb{V}_{\sL_2, \vec{\lambda}, \ell} :=  \mathbb{V}(\mathfrak{sl}_{2}, \ell, \vec{\lambda}) \text{ and } \mathbb{V}_{\sP_{2\ell}, \vec{\lambda}, 1}:= \mathbb{V}(\mathfrak{sp}_{2\ell}, 1, \vec{\lambda}).$$
\end{definition}

\subsection{Ranks of $\VsL$ and $\VsP$}
 In \cite[p.~42]{littel}, Littelmann constructs a generalized Littlewood-Richardson rule which can be used to show the following equality of ranks associated to the vector bundles in Def.~\ref{def1}. 

\begin{fact} \label{rankfact}
The vector bundles $\VsL$ and $\VsP$ in Def.~\ref{def1} have the same rank. That is,
$$\rk(\VsL) = \rk(\VsP). $$
\end{fact} 

Such an equality on the ranks can also be seen from the duality result in \cite[Cor.~3]{MUK13}. The relationship between the spaces in \cite{MUK13} (i.e., the fibers of the vector bundles) does not extend to the boundary of $\M_{0,n}$ and so the divisors we get from the bundles $\VsL$ and $\VsP$ are not always equivalent (we show many examples of this in Section \ref{examples}). The main result of this paper (Prop.~\ref{main}) gives necessary and sufficient conditions for when the divisors $c_1(\VsL)$ and $c_1(\VsP)$ \textit{are} equal.

\subsection{Degrees of $\VsL$ and $\VsP$ when $n=4$}\label{degrees}

For any simple Lie algebra $\mathfrak{g}$, one can compute the degree of $\mathbb{V}(\mathfrak{g}, \vec{\lambda}, \ell)$ with an $F$-curve by using \cite[Prop.~2.5]{Fakhruddin}. This formula computes the intersection number as a sum of products consisting of ranks of conformal blocks with fewer weights and a degree of a vector bundle on $\M_{0,4} \cong \mathbb{P}^1$. And so in order to make such a computation, one must first compute degrees of bundles with four weights. Fakhruddin has a general formula in \cite[Cor.~3.5]{Fakhruddin} for computing such a degree for any simple Lie algebra. We simplify these formulas for four pointed bundles $\VsL$ and $\VsP$ in Lemmas \ref{DegreeSL} and \ref{DegreeSP}.

In the following formulas we let $\ell$ and $\vec{\lambda} = (a, b, c, d)$ be fixed so that $\ell \geq a \geq b\geq c \geq d$ and $a+b+c+d=2(\ell+s)$ for some integer $s$. 

\begin{lemma}\label{DegreeSL}
Let $\VsL$ be defined as in Def. \ref{def1} with weights $\vec{\lambda}$ at level $\ell$. If $s\geq 0$, then
 \begin{equation}\label{sl2Degree}
 \deg(\VsL)= \rk(\VsL)\cdot s.
 \end{equation}
 
 Otherwise, $\deg(\VsL) = 0$. 
\end{lemma}

\begin{lemma} \label{DegreeSP}
Let $\VsP$ be defined as in Def. \ref{def1} with weights $\vec{\lambda}$ at level $\ell$. 
 If $ a \leq \ell+s$, then  

\begin{displaymath}
   \deg(\VsP)  = \left\{
     \begin{array}{lr}
       max \{ 0, (\ell + 1 - a)(\ell +2s - a)/2 \} & \text{ if } a+d \geq b+c \text{ and } 0 < s\\
      (\ell+s+1-a)(\ell +s-a)/2  & \text{ if } a+d \geq b+c \text{ and } 0 \geq s\\
      max \{ 0, ( 1 +d - s)(d + s)/2 \}  & \text{ if }  a+d \leq b+c \text{ and } 0 < s\\
      d(d+1)/2  & \text{ if } a+d \leq b+c \text{ and } 0 \geq s\\
     \end{array}
   \right.
\end{displaymath} 

Otherwise, $\deg(\VsP) =0$.
\end{lemma}

\begin{remark}
Variations of these formulas have been seen in previous work. The original formula, from which we obtain Lemma~\ref{DegreeSL}, first appeared in \cite[Prop.~4.2]{Fakhruddin}. Further simplifications were made by B. Alexeev and stated in \cite[Lemma~3.3]{swinarskisl2}. The formula in Lemma \ref{DegreeSP} first appeared in \cite[Prop.~5.4]{Fakhruddin} with certain conditions of the weights implicitly assumed to not guarantee trivially. We briefly justify the formula in Lemma \ref{DegreeSP} to include all cases of weights.

\end{remark}

\begin{claim}\label{correction}
Let $\vec{\lambda}=(a, b, c, d)$ and $\ell$ be nonnegative integers such that $\ell \geq a \geq b\geq c \geq d$. Suppose $2(\ell+s)= a+b+c+d$ for some integer $s.$ If $a > s + \ell$, then $\rkk(\VsP)$ is zero and thus,
$$\deg(\VsP) = 0.$$
\end{claim}

\begin{proof}
We use the Generalized Triangle Inequality in \cite[Lemma~3.8]{ags} to show $\rkk(\VsP)=0$, from this the degree result follows. The Generalized Triangle Inequality is for ranks of $\sL_2$ bundles. By Fact \ref{rankfact}, this inequality provides a condition for when $\rk(\VsP)$ is necessarily zero. The Generalized Triangle Inequality gives that if $a > b+c+d$ then $\rk(\VsP) = 0$. 

If we assume $a > \ell +s$ then since $a +b+c+d = 2(\ell + s)$, it follows immediately that $a > b+c+d$ and so $\rkk(\VsL) = \rkk(\VsP) = 0$. 


\end{proof}
By comparing the degree formula in Lemma~\ref{DegreeSP} with the rank formula in \cite[Lemma 3.3]{swinarskisl2}, we can rewrite this degree formula in terms of the rank and write the following formula analogous to that for type A in Lemma \ref{DegreeSL}.

\begin{lemma} \label{DegreeSPRank}
Let $\VsP$ be defined as in Def. \ref{def1} with weights $\vec{\lambda}$ at level $\ell$. 
 If $ a \leq \ell+s$, then  

\begin{displaymath}
   \deg(\VsP)  = \left\{
     \begin{array}{lr}
       \rk(\VsP)(\ell +2s - a)/2  & \text{ if } a+d \geq b+c \text{ and } 0 < s\\
      \rk(\VsP)(\ell +s-a)/2  & \text{ if } a+d \geq b+c \text{ and } 0 \geq s\\
      \rk(\VsP)(d + s)/2   & \text{ if }  a+d \leq b+c \text{ and } 0 < s\\
      \rk(\VsP)d/2  & \text{ if } a+d \leq b+c \text{ and } 0 \geq s\\
     \end{array}
   \right.
\end{displaymath} 

Otherwise, $\deg(\VsP) =0$.
\end{lemma}

Comparing degree terms in Lemma~\ref{DegreeSL} and Lemma~\ref{DegreeSPRank}, it follows,

\begin{corollary}\label{Inequality}
For the bundles $\VsL$ and $\VsP$ with $\vec{\lambda} = (a, b, c, d)$, the following inequality holds
$$\deg(\VsL) \leq \deg(\VsP).$$ 
\end{corollary}

\subsection{Rank one classification of $\mathbb{V}_{\sL_2, \ell}$ and $\mathbb{V}_{\sP_{2\ell},1}$ when $n=4$}

In \cite[Thm.~1.1]{Hobson} we have completely classified all $\VsL$ of rank one. Specializing to four pointed bundles and using the rank equality from Fact \ref{rankfact} we state these results in the following lemma. 


\begin{lemma}\label{RankOne}
Let $\VsL$ and $\VsP$ be defined as in Def. \ref{def1} with weights decreasing weights, $\vec{\lambda} = (a, b, c, d)$ at level $\ell$ such that $a+b+c+d = 2(\ell + s),$ for some integer $s$. 
Then $\rk(\mathbb{V}_{\sL_2, \vec{\lambda}, \ell}) = \rk(\mathbb{V}_{\sP_{2 \ell, \vec{\lambda}, 1}} )=1$ if and only if one of the following sets of conditions are satisfied:

\begin{itemize}
\begin{minipage}{0.4\linewidth}
\item[1) ] $s \geq 0,$
\item[2) ] $a, b, c, d \geq s,$ and
\item[3) ]  $a = \ell \text{ or } d=s$ or
\end{minipage}
\begin{minipage}{0.4\linewidth}
\item[1) ] $s < 0 $ and
\item[2) ] $a = \ell +s. $
\end{minipage}
\end{itemize}
 
 Furthermore, $\rk(\mathbb{V}_{\sL_2, \vec{\lambda}, \ell}) = \rk(\mathbb{V}_{\sP_{2 \ell, \vec{\lambda}, 1}} )>1$ if and only if one of the following sets of conditional are satisfied:
 
 \begin{itemize}
\begin{minipage}{0.4\linewidth}
\item[1) ] $s \geq 0,$
\item[2) ] $a, b, c, d > s,$ and
\item[3) ]  $a \neq \ell $ or
\end{minipage}
\begin{minipage}{0.4\linewidth}
\item[1) ] $s < 0 $ and
\item[2) ] $a < \ell +s. $
\end{minipage}
\end{itemize}

\end{lemma}


\subsection{Plussing for $\sL_2$ bundles}

To determine the rank of an $\sL_{r+1}$ bundle a method called ``plussing'' on the weights can be used \cite[Definition 8.2]{BGK}. We state the result of ``plussing'' for $\sL_2$ bundles and give an argument for this identity using factorization. A similar argument for $n$ even (and $I = [n]$ in the statement below) was communicated to me by Dave Swinarski.

\begin{claim}\label{plussing}
Let $\ell > 0$ be fixed and $\vec{\lambda} = ( \lambda_1, ..., \lambda_n) \in P_{\ell}(\sL_2)^n$. Let $ I \sqcup J = [n]$ be a partition into two disjoint subsets such that $|I|$ is even (and we allow $J$ to be empty). Let $\vec{\mu}$ denote the $n$-tuple of weights given by $\{ \ell - \lambda_i\}_{i\in I} \cup \{\lambda_j\}_{j \in J}$. Then $$\rkk(\mathbb{V}(\sL_2, \vec{\lambda}, \ell)) = \rkk(\mathbb{V}(\sL_2, \vec{\mu}, \ell)).$$
\end{claim}

\begin{proof}
First, suppose $\vec{\lambda} = (a, b, c)$. Then the rank of $\mathbb{V}(\sL_2, \vec{\lambda}, \ell)$ is one if and only if $a+b+c$ is even and the inequalities in \ref{fusion} are satisfied. These inequalities are often called the \textit{Three Point Fusion Rules} for $\sL_2$ \cite[Prop.~3.5]{ags}.
\begin{align*}
a &\leq b+c          \\ \numberthis \label{fusion}
b&\leq a+c     \\
c&\leq a+b   \\
a+b+ c&\leq 2\ell  
\end{align*}
In order, these inequalities imply the following set of inequalities:
\begin{align*}
(\ell -b) & \leq (\ell -a) + c          \\ \numberthis \label{fusion2}
(\ell -a) & \leq (\ell - b) +c    \\
(\ell -a) +(\ell - b) + c & \leq 2 \ell  \\
c &\leq (\ell -a) + (\ell -b),
\end{align*}
which determine, $\rkk(\mathbb{V}(\sL_2, (\ell-a, \ell-b, c), \ell)) =1$. Hence, the result is true for $n=3$ weights.

We now consider two cases determined by the parity of $n$. The case when $n$ is even will follow from our argument with $n$ odd. 

Suppose $n$ is some fixed odd integer. For our base case, we have shown the result holds for $n=3$ weights. For the induction step, we factorize using a partition of $n-2$ weights and two weights, one of which has been ``plussed.'' Suppose $i\in I$ and $j \notin I$. Let $\vec{\lambda}_{\hat{i}, \hat{j}}$ the $(n-2)$-tuple obtained from $\vec{\lambda}$ with weights in the $i^{th}$ and $j^{th}$ spot removed. In the equation below, we drop notation and fixed $\sL_2$ and level $\ell$ for all bundles. Then,
\begin{align*}
\rkk(\mathbb{V}( \vec{\lambda})) & = \sum_{\mu=0}^{\ell} \rkk(\vec{\lambda}_{\hat{i}, \hat{j}}, \mu) \rkk(\lambda_i, \lambda_j, \mu)       \\ \numberthis \label{rankodd}
& =  \sum_{\mu=0}^{\ell} \rkk(\vec{\mu}_{\hat{i}, \hat{j}}, \ell -\mu) \rkk(\ell -\lambda_i, \lambda_j, \ell -\mu) \\
 &= \rkk( \vec{\mu}), 
\end{align*}
The second equality above follows from the induction hypothesis and the case $n=3$, and the third equality follows by factorization in reverse.

Now suppose $n$ is even. We use factorization again but now partition the weights into $n-2$ weights and two weights that have both been ``plussed.'' Suppose $i, j \in I$, then
\begin{align*}
\rkk(\vec{\lambda}) & = \sum_{\mu=0}^{\ell} \rkk(\vec{\lambda}_{\hat{i}, \hat{j}}, \mu) \rkk(\lambda_i, \lambda_j, \mu)         \\  \numberthis \label{rankeven}
& = \sum_{\mu=0}^{\ell} \rkk(\vec{\mu}_{\hat{i}, \hat{j}}, \mu) \rkk(\ell -\lambda_i, \ell -\lambda_j, \mu)) \\
& = \rkk( \vec{\mu}). 
\end{align*}
The justification of the equalities is the same as the odd case (notice, we do \textit{not} need to ``plus'' the attaching weight, $\mu$, in this case).
\end{proof}

\subsection{Stabilizing Lie Rank}
We now define a term and type C divisor associated to any set of weakly decreasing nonnegative integers $\vec{\lambda} = (a_1, ..., a_n)$. 

\begin{definition}\label{rlam}
Let $\vec{\lambda} = (a_1, ..., a_n)$ be an $n$-tuple of weakly decreasing nonnegative integers such that $|\vec{\lambda}|=\sum_{i=1}^n (a_i)$ is even. We define the \textit{stabilizing Lie rank} associated to $\vec{\lambda}$ to be, $$r_{\vec{\lambda}} = \frac{\sum (a_i)}{2}-1.$$ If $r_{\vec{\lambda}} \geq a_1$ (i.e., $\vec{\lambda} \in P_{1}(\sP_{2 r_{\vec{\lambda}}})^n$), then we call $c_1(\Vrlam)$ the stable Lie divisor for $\vec{\lambda}$.
\end{definition}

\begin{remark}\label{CL}
Interpreting $\vec{\lambda} = (a_1, ..., a_n)$ in Def. \ref{rlam} as a set of $n$ dominant integral weights for $\sL_2$, one can define the \textit{critical level} of $\vec{\lambda}$ \cite[Sect.~4.3]{Fakhruddin} (see \cite[Def.~1.1]{BGMA} for a general definition of critical level associated to bundles of type $\sL_{r+1}$). Additionally, one can show that if $a_1 \geq \rlam$ (see Claim~\ref{correction}) then $\rkk(\mathbb{V}_{\sL_{2r}, \vec{\lambda}, 1})=0$ (and hence the divisor is trivial) for any integer $r$. Because of this, we will assume the weight vector $\vec{\lambda}$ is such that $a_1 \leq \rlam$. 
\end{remark}

To reduce confusion, we will refer to the rank $\ell$, of the Lie algebra $\sP_{2\ell}$, as the \textit{Lie rank} and the rank of the vector bundle $\mathbb{V}_{\sP_{2\ell}, \vec{\lambda}, 1}$ as the \textit{vector bundle rank}.


\section{Proof of divisor equivalence for rank one bundles}\label{proof}

In this section we prove Proposition \ref{main}. The set of $F$-curves span the vector space of 1-cycles on $\M_{0,n}$. We proceed by showing the divisors $c_1(\VsL)$ and $c_1(\VsP)$ intersect all $F$-curves in the same degree if and only if the bundles have rank one (or zero). By using factorization to compute ranks, we can show that the degree computation simplifies greatly in the case when the vector bundles have rank one. We state this simplification in the following observation.


\begin{observation}\label{OneTerm}
For $\mathbb{V}(\mathfrak{g}, \vec{\lambda}, \ell)$ of rank 1, there is at most one nonzero term in the intersection formula of \cite[Prop.~2.7]{Fakhruddin}. Particularly, this potential nonzero term is the degree of a four pointed bundle with rank one.
\end{observation}

From this observation and Fact \ref{rankfact}, Corollary~\ref{FacRankSLSP} for computing degrees for rank one bundles $\VsL$ and $\VsP$ follows.

\begin{corollary}\label{FacRankSLSP}
For a fixed $\ell$ and $\vec{\lambda}$, let $\VsL$ and $\VsP$ be defined as in Def. \ref{def1}. If $\rk(\VsL)$ is one, then for any $F$-curve, $F_{I_1, I_2,I_3,I_4}$, there is a four-tuple of nonnegative integers $\vec{\mu}=(a, b, c, d)$ such that

\begin{equation}
\deg(\VsL| F_{I_1, I_2,I_3,I_4}) = \deg(\mathbb{V}(\sL_2, \vec{\mu}, \ell))
\end{equation}
and
\begin{equation}
\deg(\VsP| F_{I_1, I_2,I_3,I_4}) = \deg(\mathbb{V}(\sP_{2\ell}, \vec{\mu}, 1)).
\end{equation}
\end{corollary}


It follows from Corollary \ref{FacRankSLSP} that to compare intersection numbers for rank one bundles $\VsL$ and $\VsP$ with any $F$-curve, we need only compare the four pointed degrees of rank one bundles. We now prove a result about such bundles. 

\begin{lemma}\label{RankDegree}
Let $\ell$ and $\vec{\lambda} = (a, b, c, d)$ be fixed such that $\ell \geq a \geq b\geq c \geq d$ and $a+b+c+d=2(\ell+s)$ for some integer $s$. Let $\VsL$ and $\VsP$ be defined as in Def. \ref{def1}. Then
 $$\rk(\VsL) = \rk(\VsP) = 1 \text{ or } 0$$ if and only if $$\deg(\VsL ) = \deg(\VsL) .$$ Particularly, in the case when degrees are equal, they are equal to $max \{ 0, s\}$. 
\end{lemma}

\begin{proof}
We first prove the forward ($\Rightarrow$) direction. If the rank of the bundles is zero, then degree formulas will be consistent (both trivial). Assume then that $\rkk(\VsL) = \rkk(\VsP) = 1$. We show the degree formulas in Lemma \ref{DegreeSL} and \ref{DegreeSP} are equal. To do this, we compare the corresponding degree formulas in the  four cases determined be the relationship of $a, b, c, d, \ell$ and $s$ in Lemma~\ref{DegreeSP}. We go through the first case; the other cases follow from similar calculations.


In the first case of Lemma \ref{DegreeSP}, suppose $a+d \geq b+c$ and $s \geq 0$. Since we are assuming $\rk(\mathbb{V}_{\sL_2, \ell}) = 1$, Lemma \ref{DegreeSL} and \ref{DegreeSP} give 

\begin{equation}\label{sl2degrees}
\deg(\mathbb{V}_{\sL_2, \ell}) = s \text{ and }
\end{equation}

\begin{equation}\label{CaseOne}
\deg(\mathbb{V}_{\sP_{2\ell},1}) = max \{ 0, (\ell + 1 - a)(\ell +2s - a)/2 \} 
\end{equation}

Now, since $a \leq \ell$ and $s \geq 0$ the max in Equation \ref{CaseOne} is nonzero. Furthermore, since $\rkk( \mathbb{V}_{\sP_{2\ell},1}) = 1$ Lemma \ref{RankOne} implies $d \geq s$ and we either have $a = \ell$ or $d = s$. However, since $a+b+c+d= 2(\ell +s)$ and $a+d \geq b+c$ it follows that $a+d \geq \ell + s$  and so indeed, $a = \ell$. Using this, Equation \ref{CaseOne} becomes, 
$$= (\ell + 1 - \ell)(\ell + 2s - \ell) /2 = s,$$
showing Equations \ref{sl2degrees} and \ref{CaseOne} are equal and so $\deg(\mathbb{V}_{\sP_{2\ell},1}) =\deg(\mathbb{V}_{\sL_2, \ell})$. 

%
%
%
%



To show equality in the other three cases ($a+d < b+c$ and $s \geq 0$,  $a+d \geq b+c$ and $s < 0$, and  $a+d < b+c$ and $s < 0$) follow similar arguments and computations. \\

For the reverse implication ($\Leftarrow$), we assume $\rk(\VsL) = \rk(\VsP) > 1$. From Lemma \ref{RankOne} we can assume $s >0$ and both $a < \ell$ and $d > s$. We compare the four point degree formula from \cite[Prop.~4.2]{Fakhruddin} for $\sL_2$ bundles with our corresponding formula for $\sP_{2\ell}$ in Lemma~\ref{DegreeSP}. We must consider two cases.\\
 
 \textbf{Case One:} $a+d \geq b+c$. In this case, we compare, 
 
  \begin{equation}
 \deg(\VsL) =  max \{ 0, (\ell +1 -a)(s) \} \text{ and }
 \end{equation}
 \begin{equation}
 \deg(\VsP) =  max \{ 0, (\ell +1 -a)(\ell+2s-a)/2 \}.
 \end{equation}
 With our assumptions, it follows that both values are nonzero and not equal since $\ell < a$.
 
 \textbf{Case Two:} $a+d < b+c$. In this case, we compare,
  
 \begin{equation}
 \deg(\VsL) =  max \{ 0, (1+d-s)(s) \} \text{ and }
 \end{equation}
 \begin{equation}
 \deg(\VsP) =  max \{ 0, (\ell +1 -d)(d+s)/2 \}.
 \end{equation}
 With our assumptions, both values are nonzero (as $d > s$) and not equal. 
\end{proof}

The result of Lemma~\ref{RankDegree} and the method of computing rank using factorization in the discussion of Observation \ref{OneTerm} allow us to explicitly determine when two intersection numbers for $\VsL$ and $\VsP$ with an arbitrary number of $n$ weights are equal. With Corollary~\ref{Inequality} we also obtain an inequality on the intersection numbers for these bundles. We summarize this result in the following corollary. 

\begin{corollary}\label{RORF}
Let $\VsL$ and $\VsP$ be defined as in Def. \ref{def1} for some fixed integer $\ell$ and $n$-tuple $\vec{\lambda}$. Given a partition $\{1, ..., n\} = I_1 \sqcup I_2 \sqcup I_3 \sqcup I_4$ determining an $F$-curve on $\M_{0,n}$, the intersection numbers of $\VsL$ and $\VsP$ are equal on that $F$-curve if and only if the four pointed bundles appearing as the degree terms \cite[Prop.~2.7]{Fakhruddin} are rank one or zero bundles. Furthermore, the following relation always holds $$\deg(\VsL| F_{I_1, I_2,I_3,I_4}) \leq \deg(\VsP| F_{I_1, I_2,I_3,I_4}).$$
\end{corollary}

It turns out, if a vector bundle $\VsL$ or $\VsP$ satisfies the conditions in Corollary~\ref{RORF} with every possible partition determining an $F$-curve then this also implies the bundles $\VsL$ and $\VsP$ have rank one. We show this in the following lemma using a restated, but equivalent, condition. 

\begin{proposition}\label{RORFLemma}
Let $\VsL$ and $\VsP$ be defined as in Def. \ref{def1} for some fixed $\ell$ and $\vec{\lambda}$. If the rank computation along any partition $\{1, ..., n\} = I_1 \sqcup I_2 \sqcup I_3 \sqcup I_4$ determined by any $F$-curve on $\M_{0,n}$ has rank one on all of the four pointed bundles associated to the four attaching weights appearing in the degree formula \cite[Prop.~2.7]{Fakhruddin} (see Equation \ref{FacRank1} for where the bundles with four attaching weights appear explicitly), then $\rkk(\VsL) = 1$. 
\end{proposition}

\begin{proof}
We first show the result with $n=5$ weights. Using induction on the number of weights, we conclude the result for arbitrary $n$. 

For $n=5$ we will denote our weights $\vec{\lambda} = (a, b, c, d, e)$ to match notation from Lemma~\ref{RankOne}. Using a method called `plussing' (see \cite[Definition 8.2]{BGK}), we can assume at least three of our weights are $a, b, c \leq \frac{\ell}{2}$. Further, suppose $a \geq b \geq c$. Consider the partition $[n] = \{1\} \sqcup \{2\} \sqcup \{3\} \sqcup \{4, 5\}.$ Using factorization along this partition, we have the following computation (note that the involution on the $\sL_2$ weight system $\mu \mapsto \mu^*$ is the identity, so we use $\mu$ in the calculation below) 
\begin{equation}\label{FacRank1}
\rkk(\VsL) = \sum_{\vec{\mu}\in P_{\ell}(\mathfrak{sl}_2)^4} \rkk(\mathbb{V}_{\vec{\mu}})\rkk(\mathbb{V}_{a, \mu_1})\rkk(\mathbb{V}_{b, \mu_2})\rkk(\mathbb{V}_{c, \mu_3})\rkk(\mathbb{V}_{d, e, \mu_4}),
\end{equation}
where $\vec{\mu}=(\mu_1, \mu_2, \mu_3, \mu_4)$ and we denote $\mathbb{V}_{\vec{\mu}}$ for the bundle $\mathbb{V}(\sL_2, \vec{\mu}, \ell)$.

Using the fusion rules for $\sL_2$, the two pointed rank terms are nonzero (and equal to one) if and only if the two weights are equal. This equation thus reduces to,
\begin{equation}\label{FacRank1sim}
\rkk(\VsL) = \sum_{\mu \in P_{\ell}(\mathfrak{sl}_2)} \rkk(\mathbb{V}_{a, b, c, \mu})\rkk(\mathbb{V}_{d, e, \mu}).
\end{equation}

The assumption of the Proposition statement is that $\rkk(\mathbb{V}_{a, b, c, \mu}) =1$ or $0$ for any $\mu \in P(\mathfrak{sl}_2, \ell)$. In order to show the original bundle has rank one, we must show that there is only one nonzero term appearing in Equation~\ref{FacRank1sim}. 

For contradiction, suppose there are at least two nonzero terms in Equation~\ref{FacRank1sim}, that is,
\begin{equation}\label{Contradiction}
\rkk(\VsL) \geq \rkk(\mathbb{V}_{a, b, c, \mu})\rkk(\mathbb{V}_{d, e, \mu}) +  \rkk(\mathbb{V}_{a, b, c, \tilde{\mu}})\rkk(\mathbb{V}_{d, e, \tilde{\mu}}) \geq 2
\end{equation}
and $\mu \neq \tilde{\mu}$; particularly, let's assume $\mt < \mu$.

We consider possible cases of weights $\mu$ and $\tilde{\mu}$ such that $\rkk(\mathbb{V}_{a, b, c, \mu}) =  \rkk(\mathbb{V}_{a, b, c, \tilde{\mu}}) =1$ from the classification of rank one bundles from \cite{Hobson} restated for bundles with four weights in Lemma~\ref{RankOne}.

Let $s$ and $\tilde{s}$ be integers such that $$a+b+c+\mu = 2(\ell + s) \text{ and }$$ $$a+b+c+\tilde{\mu} = 2(\ell + \tilde{s}).$$ Assuming $\tilde{\mu} < \mu$, parity provides the further relation, $\mt +1 < \mu$. From this $\tilde{s} < s$ follows and additionally,
\begin{equation}\label{stilde}
s = \st - \frac{\mu - \mt}{2}
\end{equation}
We consider the following three cases: \\

\noindent \textbf{Case One:} $\tilde{s} < s < 0:$

By the rank one classification in Lemma~\ref{RankOne}, the second collection of conditions must be satisfied for the weights $\{a, b, c, \mu\}$ and $\{a, b, c, \tilde{\mu} \}.$ Hence, the largest weight in each set must be equal to $\ell + s$ or $\ell + \tilde{s}$ respectively. Since the weights are ordered $a \geq b \geq c$, the largest weights in either set are either $a,$ $\m,$ or $\mt$. If the largest weight in each collection was $a$ this would contradict $\st < s.$ And similarly, if the largest weights in each collection were $\mu$ and $\mt$, for that would imply $a+b+c= \ell + s = \ell + \st$. Hence, we must have the largest of the attaching weights be $\mu = \ell + s$ and so $a = \ell + \st$. 


We further consider two subcases:

\noindent \textbf{Case One (a):} $\tilde{\mu}+2 < \mu:$ From this assumption, it follows that $\st + 1< s$. Consider the weight $\mt +2 \in P_{\ell}(\sL_2)$ in the sum of Equation~\ref{FacRank1sim}. Since $\rk(\mathbb{V}_{d, e, \m}) =\rk(\mathbb{V}_{d, e, \mt}) =1$, then the fusion rules for $\sL_2$ shows $\rk(\mathbb{V}_{d, e, \mt +2}) =1.$ Furthermore, we consider $\rk(\mathbb{V}_{a, b, c, \mt+2})$. We have $$a + b + c + (\mt + 2) =  2(\ell + \st) +2 =2(\ell + \st +1).$$ The assumption of the Proposition statement is that this rank is one or zero. And so, again, from Lemma~\ref{RankOne}, the largest weight must be equal or larger than $\ell + \st +1$. However, since $\mt < a$, the largest weight is $a = \ell + \st < \ell + \st +1$, showing $\rk(\mathbb{V}_{a, b, c, \mt+2})>1$ which contradicts our Proposition assumption.

\noindent \textbf{Case One (b):} $\tilde{\mu}+2 = \mu:$ From this, it follows that $s = \st + 1$. We have previously shown $\m= \ell + s$ and $\mt < a = \ell + \st$. This implies $$\ell + \st > \mt = \m -2 = \ell + s -2,$$ and so $$\st + 2 \geq s,$$ contradicting $s = \st +1$. \\

\noindent \textbf{Case Two:} $\tilde{s} < 0 \leq s$:

First, consider $\rk(\mathbb{V}_{a, b, c, \m})=1.$ Using the classification of rank one from Lemma~\ref{RankOne}, we have that $ a, b, c, \m \geq s$ and one of the following are satisfied 

\begin{itemize}
\item[a)] $c = s$, $\m > s$ ($c$ is smallest weight),
\item[b)] $c \geq \m = s,$ ($\mu$ is smallest weight),
\item[c)] $\m = \ell$ ($\mu$ is largest weight).
\end{itemize}
Suppose (a) held. Then since $a + b + c + \mu = 2 \ell + 2s$ it follows $$a + b + \mu = 2 \ell + s.$$ As $ \ell \geq a + b$, this would give $$ \ell + \m \geq 2 \ell + s$$ forcing $\mu  = \ell$ and $s = 0$ (that is, condition (c) is satisfied) and so we can ignore this as a separate case.

Suppose (b) held. Then $ \mu = s$ and $\mu$ is the smallest weight. This again leads to a contradiction, as $3\frac{\ell}{2} \geq a + b + c$, but $\m =s$ forces $a + b+ c = 2\ell + s$. This provides, $$ 3\frac{\ell}{2} \geq 2 \ell + s,$$ contradicting $s \geq 0$. 

We are thus left to assume condition (c), $\mu = \ell$. With this assumption, it follows, 
\begin{equation}\label{abc}
a + b + c = \ell + 2s.
\end{equation} 
Now then, consider $\rk(\mathbb{V}_{a, b, c, \mt})=1.$ Again, using the classification of rank one from Lemma~\ref{RankOne}, we must have the largest weight of $\{ a, b, c, \mt \}$ equal to $\ell + \st$.

 We consider two cases for this largest weight. 

\noindent \textbf{Case Two (a):} Largest weight $a = \ell + \st$. With this assumption and $ a \leq  \frac{\ell}{2}$, we have $a = \ell + \st \leq \frac{\ell}{2}$ forcing $\st \leq - \frac{\ell}{2}$.  Using Equation~\ref{stilde}, it follows that,
\begin{equation}\label{relation}
s = \st + \frac{\mu - \mt}{2}  \leq - \frac{\ell}{2} + \frac{\mu - \mt}{2}.
\end{equation}
Since we have shown $\mu = \ell$, this inequality becomes $$ \geq 2s \leq -\mt.$$ Since the weights are nonnegative, this forces $\mt = s = 0$. From Equation~\ref{abc}, using $s=0$, we have $$a + b+ c = \ell.$$ Additionally, $$\ell = a+b+c+0 = a+b+c+\mt=2\ell + 2\st,$$ showing $\st = \frac{-\ell}{2}$. And so it further follows that $a= \frac{\ell}{2}$ and $b+c = \frac{\ell}{2}$.

We have shown $\mu = \ell$ and $\mt = 0$ are two attaching weights appearing as nonzero terms in Equation~\ref{FacRank1sim}, hence, $\rk(\mathbb{V}_{d, e, \ell}) =\rk(\mathbb{V}_{d, e, 0}) = 1$. Thus, the inequalities from the fusion rules for these three pointed bundles (see Equation~\ref{fusion}) are satisfied for $d, e, 0$ and $d, e, \ell$. Particularly, from $\rk(\mathbb{V}_{d, e, 0}) = 1$ we have $d \leq e + 0$ and $e \leq d + 0$ forcing $d=e$. From $\rk(\mathbb{V}_{d, e, \ell}) = 1$ we have $d + e + \ell \leq 2\ell$ and $\ell \leq d +e$ forcing $d + e = \ell$. From this, we must have $d = e = \frac{\ell}{2}$. Hence, in this case, our weights are $\vec{\lambda} = (\frac{\ell}{2}, b, c, \frac{\ell}{2}, \frac{\ell}{2})$ with $b + c = \frac{\ell}{2}$.

Consider the rank computation using factorization along the partition $\{1\}~\sqcup~\{2,3\}~\sqcup~\{4\}~\sqcup~\{5\}$. This gives the sum,
\begin{equation}\label{FacRankL}
\rkk(\VsL) = \sum_{\mu \in P_{\ell}(\mathfrak{sl}_2)} \rkk(\mathbb{V}_{\frac{\ell}{2}, \frac{\ell}{2}, \frac{\ell}{2}, \mu})\rkk(\mathbb{V}_{b, c, \mu}),
\end{equation}
where by the assumption in the Proposition statement $\rkk(\mathbb{V}_{\frac{\ell}{2}, \frac{\ell}{2}, \frac{\ell}{2}, \mu})\rkk(\mathbb{V}_{b, c, \mu}) =1$ or 0. We now show this is not true for the weight $\mu = \frac{\ell}{2}$. 

First, consider $\rkk(\mathbb{V}_{b, c, \frac{\ell}{2}}) $. With $b+c = \frac{\ell}{2}$, it is straight forward to check all inequalities with $b, c, \frac{\ell}{2}$ in Equation~\ref{fusion} are satisfied so $ \rkk(\mathbb{V}_{b, c, \mu}) =1$.

Now consider $\rkk(\mathbb{V}_{\frac{\ell}{2}, \frac{\ell}{2}, \frac{\ell}{2}, \frac{\ell}{2}}).$ Here, we have $$\frac{\ell}{2} + \frac{\ell}{2} + \frac{\ell}{2} + \frac{\ell}{2} = 2\ell.$$ Using Lemma~\ref{RankOne}, we see that each weight is $\frac{\ell}{2} > 0$ and the largest weight is $\frac{\ell}{2} < \ell$. Lemma~\ref{RankOne} concludes $\rkk(\mathbb{V}_{\frac{\ell}{2}, \frac{\ell}{2}, \frac{\ell}{2}, \frac{\ell}{2}})> 1$. \\

\noindent \textbf{Case Two (b):}  Now suppose $\mt = \ell + \st$ was the largest weight in $\{a, b, c, \mt\}$. Since $\m = \ell$, we have $a+b+c = \ell + 2s$. If $\mt = \ell + \st$, it would follows that $a+b+c = \ell + \st < \ell$ (since we are assuming $\st < 0$. This contradicts $s \geq 0$. \\

\noindent \textbf{Case Three:} $0 \leq \tilde{s} < s:$

We have already seen from Case Two, that if we assume $a, b, c \leq \frac{\ell}{2}$ and $s \geq 0$, then to have $\rk(\mathbb{V}_{a, b, c, \m})=1$ it must be that $\mu = \ell$. Such is now also for both the bundles $\mathbb{V}_{a, b, c, \m}$ and $\mathbb{V}_{a, b, c, \mt}$ since we are assuming both $s, \st \geq 0$. This shows $\mu = \mt = \ell$ contradicting our assumption that $\mt < \mu$. 

We have shown Equation~\ref{FacRank1sim} consist of one nonzero term. Particularly, since this term is a product of two ranks each of rank one, the sum is one and hence, $\rkk(\VsL) = 1.$

We now show the Proposition statement is true for an arbitrary number of weights, $\vec{\lambda}=~(\lambda_1,~\lambda_2,~...,~\lambda_n)$. Assume the statement holds for any number of weights less than $n$ and consider the factorization sum using the partition $\{1\} \sqcup \{2\} \sqcup \{3\} \sqcup \{4, 5, 6, ..., n\}$ with $\lambda_1, \lambda_2, \lambda_3 \leq \frac{\ell}{2}$. Using this partition, the rank factorizes as
\begin{equation}\label{FacRank1sim2}
\rkk(\VsL) = \sum_{\mu \in P_{\ell}(\mathfrak{sl}_2)} \rkk(\mathbb{V}_{\lambda_1, \lambda_2, \lambda_3, \mu})\rkk(\mathbb{V}_{\lambda_4, \lambda_5, \lambda_6, ..., \lambda_n, \mu}).
\end{equation}
The previous argument with $n=5$ weights shows that if a bundle satisfies the conditions in the Proposition statement, then there is only one term appearing in the sum of Equation~\ref{FacRank1sim2}. In the argument with $n=5$ weights, the term $\rkk(\mathbb{V}_{d, e, \mu})$ appears only in the second case. If this term contained more than three weights, we could further factorize and obtain $$\rkk(\mathbb{V}_{\lambda_4, \lambda_5,..., \lambda_n, \mu}) = \sum_{\nu \in P_{\ell}(\sL_2)} \rkk(\mathbb{V}_{ \mu, \lambda_4, \nu}) \rkk(\mathbb{V}_{ \nu, \lambda_5, \lambda_6, ...\lambda_n, \nu}).$$ The argument from the $n=5$ case would force $\lambda_4 = \nu = \frac{\ell}{2}$. The argument would then follow using the partition of the initial $n$ weights into $\{1\} \sqcup \{2, 3\} \sqcup \{4\} \sqcup \{5, 6, ..., n\}$. The term, $\rkk(\mathbb{V}_{a, d, \nu, \mu})= \rkk(\mathbb{V}_{\frac{\ell}{2}, \frac{\ell}{2}, \frac{\ell}{2}, \frac{\ell}{2}})>1$ would appear in the factorization sum for the bundle $\mathbb{V}_{\vec{\lambda}}$, again, contradicting the assumption of this Proposition.

It then follows that the sum in Equation~\ref{FacRank1sim2} reduces to one term,
\begin{equation}\label{FacRank1sim3}
\rkk(\VsL) = \rkk(\mathbb{V}_{\lambda_1, \lambda_2, \lambda_3, \mu})\rkk(\mathbb{V}_{\lambda_4, \lambda_5, \lambda_6, ..., \lambda_n, \mu}).
\end{equation}
We must show that this term is one. Particularly, since we are assuming $ \rkk(\mathbb{V}_{\lambda_1, \lambda_2, \lambda_3, \mu}) = 1$, we must show $$\rkk(\mathbb{V}_{\lambda_4, \lambda_5, \lambda_6, ..., \lambda_n, \mu}) = 1.$$ We proceed by showing the bundle $\mathbb{V}_{\lambda_4, \lambda_5, \lambda_6, ..., \lambda_n, \mu}$ satisfies the condition in the Proposition statement. This bundle now has $n-2 < n$ weights. By our inductive assumption, we will conclude this bundle has rank one. 

Consider any partition of $\{4, 5, ..., n\} = I_1 \sqcup I_2 \sqcup I_3 \sqcup I_4$. We want to show that the four weight bundles appearing in the rank factorization sum have rank one. Using this partition (and including the attaching weight $\mu$ in the first set) we get, 

\begin{equation}\label{FacRank2}
\rkk(\mathbb{V}_{\lambda_4, \lambda_5, ..., \lambda_n, \mu}) = \sum_{\vec{\mu}\in P_{\ell}(\mathfrak{sl}_2)^4} \rkk(\mathbb{V}_{\vec{\mu}})\rkk(\mathbb{V}_{\lambda_{I_1}, \mu, \mu_1})\rkk(\mathbb{V}_{\lambda_{I_2}, \mu_2})\rkk(\mathbb{V}_{\lambda_{I_3}, \mu_3})\rkk(\mathbb{V}_{\lambda_{I_4}, \mu_4})
\end{equation}
where $\vec{\mu} = (\mu_1, ..., \mu_4)$ and $\lambda_{I_i} = \{ \lambda_j \}_{j \in I_i}$. Now, since $\rkk(\mathbb{V}_{\lambda_1, \lambda_2, \lambda_3, \mu}) =1$, we can multiply this to each term in the sum, to obtain,

\begin{equation}\label{FacRank3}
\rkk(\mathbb{V}_{\lambda_4, \lambda_5, \lambda_6, ..., \lambda_n, \mu}) = \sum_{\vec{\mu}\in P_{\ell}(\mathfrak{sl}_2)^4} \rkk(\mathbb{V}_{\lambda_1, \lambda_2, \lambda_3, \mu}) \rkk(\mathbb{V}_{\vec{\mu}})\rkk(\mathbb{V}_{\lambda_{I_1}, \mu, \mu_1})\rkk(\mathbb{V}_{\lambda_{I_2}, \mu_2})\rkk(\mathbb{V}_{\lambda_{I_3}, \mu_3})\rkk(\mathbb{V}_{\lambda_{I_4}, \mu_4}).
\end{equation}

Furthermore, we can show 
 $$\rkk(\mathbb{V}_{\lambda_1, \lambda_2, \lambda_3, \lambda_{I_1}, \mu_1}) = \rkk(\mathbb{V}_{\lambda_1, \lambda_2, \lambda_3, \mu}) \rkk(\mathbb{V}_{\lambda_{I_1}, \mu, \mu_1}).$$ If this did not hold, then there would be some other weight $\mt$ such that  $$\rkk(\mathbb{V}_{\lambda_1, \lambda_2, \lambda_3, \mt})\rkk(\mathbb{V}_{\lambda_{I_1}, \mt, \mu_1}) \neq 0.$$ From this, it would follow that $\mt$ was also an attaching weight in the original factorization sum of Equation~\ref{FacRank1sim2}, contradicting Equation~\ref{FacRank1sim3}. 
  
Thus, Equation~\ref{FacRank3} can be written as,
\begin{equation}\label{FacRank4}
\sum_{\vec{\mu}\in P_{\ell}(\mathfrak{sl}_2)^4} \rkk(\mathbb{V}_{\vec{\mu}}) \rkk(\mathbb{V}_{\lambda_1, \lambda_2, \lambda_3, \lambda_{I_1}, \mu_1}) 
\rkk(\mathbb{V}_{\lambda_{I_2}, \mu_2})\rkk(\mathbb{V}_{\lambda_{I_3}, \mu_3})\rkk(\mathbb{V}_{\lambda_{I_4}, \mu_4}).
\end{equation}
 
Computing the rank of the original bundle $\VsL$ using the partition of $[n] = (\{1, 2, 3 \} \cup I_1) \sqcup I_2 \sqcup I_3 \sqcup I_4$, provides this same rank decomposition. Explicitly,

\begin{align*}
\rkk(\VsL) &= \rkk(\mathbb{V}_{\lambda_1, \lambda_2, \lambda_3, \mu}) \rkk(\mathbb{V}_{\lambda_4, \lambda_5, \lambda_6, ..., \lambda_n, \mu})  \\
 &=  \rkk(\mathbb{V}_{\lambda_1, \lambda_2, \lambda_3, \mu}) \Big( \sum_{\vec{\mu}\in P_{\ell}(\mathfrak{sl}_2)^4}  \rkk(\mathbb{V}_{\vec{\mu}})\rkk(\mathbb{V}_{\lambda_{I_1}, \mu, \mu_1})\rkk(\mathbb{V}_{\lambda_{I_2}, \mu_2})\rkk(\mathbb{V}_{\lambda_{I_3}, \mu_3})\rkk(\mathbb{V}_{\lambda_{I_4}, \mu_4})\Big) \\
  &=  \sum_{\vec{\mu}\in P_{\ell}(\mathfrak{sl}_2)^4} \rkk(\mathbb{V}_{\lambda_1, \lambda_2, \lambda_3, \mu}) \rkk(\mathbb{V}_{\vec{\mu}})\rkk(\mathbb{V}_{\lambda_{I_1}, \mu, \mu_1})\rkk(\mathbb{V}_{\lambda_{I_2}, \mu_2})\rkk(\mathbb{V}_{\lambda_{I_3}, \mu_3})\rkk(\mathbb{V}_{\lambda_{I_4}, \mu_4}) \\
    &=  \sum_{\vec{\mu}\in P_{\ell}(\mathfrak{sl}_2)^4}  \rkk(\mathbb{V}_{\vec{\mu}})\Big( \rkk(\mathbb{V}_{\lambda_1, \lambda_2, \lambda_3, \mu})  \rkk(\mathbb{V}_{\lambda_{I_1}, \mu, \mu_1}) \Big) \rkk(\mathbb{V}_{\lambda_{I_2}, \mu_2})\rkk(\mathbb{V}_{\lambda_{I_3}, \mu_3})\rkk(\mathbb{V}_{\lambda_{I_4}, \mu_4}) \\
 &= \sum_{\vec{\mu}\in P_{\ell}(\mathfrak{sl}_2)^4} \rkk(\mathbb{V}_{\vec{\mu}}) \rkk(\mathbb{V}_{\lambda_1, \lambda_2, \lambda_3, \lambda_{I_1}, \mu_1}) 
\rkk(\mathbb{V}_{\lambda_{I_2}, \mu_2})\rkk(\mathbb{V}_{\lambda_{I_3}, \mu_3})\rkk(\mathbb{V}_{\lambda_{I_4}, \mu_4}),
\end{align*}
The assumption of the Proposition is that all bundles $\rkk(\mathbb{V}_{\vec{\mu}})$ appearing in this sum have rank one or zero. These are the same four pointed bundles appear in Equation~\ref{FacRank2}. We can conclude that the assumption of the Proposition statement is satisfied for the $n-2$ weight bundle $\mathbb{V}_{\lambda_4, \lambda_5, \lambda_6, ..., \lambda_n, \mu}$. By our inductive assumption, the rank of this bundle is one.

This concludes the proof of the Proposition.

\end{proof}

\begin{remark}
In Observation \ref{OneTerm} we discussed the converse of Proposition~\ref{RORFLemma} (i.e., for a vector bundle of rank one, when we compute the rank using factorization along any partition of $\{ 1, ..., n\}$ determined by an $F$-curve, the sum becomes one term equal to one).
\end{remark}

%
%
%

We summarize our results of this section with the proof of our main result.

\begin{proof}[Proof of Proposition \ref{main}]
Let $\VsL$ and $\VsP$ be defined as in Def. \ref{def1} for a fixed integer $\ell$ and $n$-tuple $\vec{\lambda}$. Observation \ref{OneTerm} and Proposition~\ref{RORFLemma} shows that such bundles have rank one if and only if the rank calculated by factorizing along the partition of the $n$ weights determined by any $F$-curve has rank one on the four pointed bundles associated to the four attaching weights. By Lemma \ref{RankDegree}, the degrees of four pointed bundles $\VsL$ and $\VsP$ are equal if and only if the corresponding bundles have rank one or zero, and otherwise, by Corollary~\ref{Inequality} the degree term for the $\VsP$ bundle is larger. Hence, such bundles will have equal intersection on every $F$-curve if an only if the bundles have rank one or zero.

\end{proof}


\section{Generalized Veronese quotients and maps given by $c_1(\VsP)$}\label{Veronese}

There are birational models of $\M_{0,n}$ given by so called \textit{generalized Veronese quotients}, $V^d_{\gamma, \mathcal{A}}$. These projective varieties parametrize configurations of $n$ weighted points lying on (limits of) weighted Veronese curves of degree $d$ in projective $d$ space. They were first constructed in \cite{Gian} with $S_n$-invariant weights $\mathcal{A}$ on the $n$ marked points and weight $\gamma=0$ on the underlying curve. They were later generalized in \cite{GG, GJM}. These moduli spaces receive birational morphism from $\M_{0,n}$ and are constructed as GIT quotients generalizing Kapranov's birational model of $\M_{0,n}$ give by $(\mathbb{P}^1)^n // SL(2)$ in \cite{Kap}.

In case $\gamma=0$ the birational contractions $$ \varphi_{0, \mathcal{A}} : \M_{0,n} \rightarrow V^d_{0, \mathcal{A}}$$ are known to correspond to conformal blocks divisors in type A at level 1, $c_1(\mathbb{V}(\sL_{r+1}, \vec{\lambda}, 1))$ \cite[Thm.~3.2]{GG}. The higher level divisors $c_1(\mathbb{V}(\sL_{2}, (\omega_1)^n, \ell))$ are also know to give contractions with $\varphi_{\frac{\ell-1}{\ell+1}, \mathcal{A}}$  \cite{GJMS13}.

We showed in \cite[Thm.~5.1]{Hobson} that for $\VsL$ of rank one we can explicitly write $c_1(\VsL)$ as a sum of divisors for $\sL_2$ at level one. Using this decomposition, the description of the maps from divisors at level one in \cite[Thm.~3.2]{GG}, and Prop.~\ref{main}, we obtain the following result about the maps from the divisors $c_1(\VsL)=  c_1(\VsP)$ when $\rkk(\VsL) = \rkk(\VsP) = 1$.

\begin{claim}\label{Maps}
Let $\VsL$ and $\VsP$ be defined as in Def.~\ref{def1}. Let $|\vec{\lambda}| = 2(d \ell + p)$ for some $d \geq 0$ and $\ell > p \geq 0$. If $\rkk(\VsL) = \rkk(\VsP) =1$, then the contractions $\phi_{\mathbb{D}}$ given by $\mathbb{D}=c_1(\VsL)=  c_1(\VsP)$ maps to a product of $\ell$ generalized Veronese quotients, $$\phi_{\mathbb{D}}: \M_{0,n} \rightarrow \prod_{i=1}^{p} V^{d}_{0, (\frac{1}{2})^{2d+2}} \times \prod_{j=p+1}^{\ell} V^{d-1}_{0, (\frac{1}{2})^{2d}}.$$ 
\end{claim}

\begin{remark}
The map $\phi_{\mathbb{D}}$ first factors through a product of $\ell$ forgetful maps, where in each factor we have forgotten all but either $2d+2$ or $d$ marked points. By a dimension count, the map $\phi_{\mathbb{D}}$ is not surjective to this product. Furthermore, the following result shows the map $\phi_{\mathbb{D}}$ does not factor through a smaller product.
\end{remark}

\begin{claim}
The map $\phi_{\mathbb{D}}$ in Claim \ref{Maps} does not factor through a smaller product.
\end{claim}

\begin{proof}
We show that if $c_1(\VsP)=c_1(\VsL) = \sum_{i=1}^{\ell} c_1(\VsLi)$ then no smaller linear combination of the divisors are linearly equivalent to $c_1(\VsP)$.

We can argue by contradiction. Suppose, $c_1(\VsL) = \sum_{i=1}^{\ell-1} m_i c_1(\VsLi)$ with $m_i \geq 1$. Thus, we have $$  c_1(\VsL) =\sum_{i=1}^{\ell} c_1(\VsLi)= \sum_{i=1}^{\ell-1} m_i c_1(\VsLi) = \sum_{i=1}^{\ell-1} (m_i -1) c_1(\VsLi) + \sum_{i=1}^{\ell-1}c_1(\VsLi).$$ 
Hence, $$ \sum_{i=1}^{\ell}c_1(\VsLi)=  \sum_{i=1}^{\ell-1} (m_i-1) c_1(\VsLi) + \sum_{i=1}^{\ell-1}c_1(\VsLi).$$ 
Subtracting the term on the right gives the relationship, $$c_1(\VsLl)=  \sum_{i=1}^{\ell-1} (m_i-1) c_1(\VsLi) .$$ For any $c_1(\VsLi)$ and $c_1(\VsLj)$ in the original sum, we can always find an $F$-curve such that the intersection deg$(\VsLi|F) =0$ and deg$(\VsLj|F) =0.$ Particularly, we can find such an $F$-curve with this behavior for the divisors $c_1(\VsLl)$ and $c_1(\VsLi)$ with $m_i > 1.$ Intersecting both sides of the assumed equation provides a contradiction. 
\end{proof}
We conclude with an explicit example to describe the product which $\phi_{\mathbb{D}}$ maps to.

\begin{example}
Let $\VsP = \mathbb{V}(\sP_{2\cdot9}, (\omega_9, \omega_8, \omega_8, \omega_8, \omega_8, \omega_8, \omega_2, \omega_1), 1)$. It was shown in \cite[Example~5.3]{Hobson} that the bundle $\VsL= \mathbb{V}(\mathfrak{sl}_2, (9\omega_1, 8\omega_1, 8\omega_1, 8\omega_1, 8\omega_1, 8\omega_1, 8\omega_1, 2\omega_1, 1\omega_1), 9)$ has rank one and the first Chern class decomposes as a sum of first Chern classes of $\sL_2$ bundles at level 1. That is, $c_1(\VsL)$ decomposes into a sum $$c_1(V_{(1,1,1,1,1,1,1,1,0)}) + c_1(V_{(1,1,1,1,1,1,1,1,0}) + c_1(V_{(1,1,1,1,1,1,1,0,1)}) + $$ $$c_1(V_{(1,1,1,1,1,1,0,0,0)})  + c_1(V_{(1,1,1,1,1,0,1,0,0)})  + c_1(V_{(1,1,1,1,0,1,1,0,0)}) +$$ $$ c_1(V_{(1,1,1,0,1,1,1,0,0)}) +  c_1(V_{(1,1,0,1,1,1,1,0,0)}) + c_1(V_{(1,0,1,1,1,1,1,0,0)}).$$ 
The subscript of these bundles denotes the weights. For example, $V_{(1,1,1,1,1,0,1,0,0)}$ is the $\sL_2$ bundle at level one and weights $(\omega_1,\omega_1, \omega_1, \omega_1, \omega_1, 0, \omega_1, 0, 0)$. Such a decomposition is seen from the column data of the unique Young tableau formed to compute rank($\VsL$) (see Figure \ref{figTab}). In this Young tableau, each column corresponds to a level one first Chern class in the sum, where a nonzero weight appears in the first Chern class for each entry appearing in a that column. The image of such a map is then into a product of generalized Veronese quotients as in Claim \ref{Maps} where in each component, we have forgotten the zero weights determined by the corresponding level one first Chern class in the decomposition. By our main result, Proposition \ref{main}, the first Chern class $c_1(\VsL)$ is equal to $c_1(\VsP)$ and similarly for the first Chern classes in the decomposition. We can then conclude that the contraction of $\M_{0,9}$ from $c_1(\VsP)$ is to the same product of generalized Veronese quotients. 

\begin{figure}[h]
\setlength{\unitlength}{0.14in} 
\centering 
$$
\ytableausetup
{mathmode, boxsize=1em}
  \begin{ytableau}
\text{1} & \text{1} & \text{1} & \text{1} & \text{1} & \text{1} & \text{1} & \text{1} & \text{1} \\
\text{2} & \text{2} & \text{2} & \text{2} & \text{2} & \text{2} & \text{2} & \text{2} & \text{3} \\
\text{3} & \text{3} & \text{3} & \text{3} & \text{3} & \text{3} & \text{3} & \text{4} & \text{4} \\
\text{4} & \text{4} & \text{4} & \text{5} & \text{4} & \text{4} & \text{5} & \text{5} & \text{5} \\
\text{5} & \text{5} & \text{5} & \text{5} & \text{5} & \text{6} & \text{6} & \text{6} & \text{6} \\
\text{6} & \text{6} & \text{6} & \text{6} & \text{7} & \text{7} & \text{7} & \text{7} & \text{7} \\
 \text{7} & \text{7} & \text{7} \\
 \text{8} & \text{8} & \text{9} 
\end{ytableau} 
$$
\caption{The unique Young tableau determining $\rkk(\VsL)$} 
\label{figTab} 
\end{figure}

\end{example}





\section{Ranks of $\sL_2$ bundles below critical level}\label{sl2ranks}

We now go through a brief interlude to show several results on ranks of conformal blocks bundles for $\sL_2$ with a fixed set of weights as the level defining the bundle varies. We frame this discussion for the Lie algebra $\sL_2$ and level $\ell$ due to its significance in previous work (\cite{swinarskisl2}, \cite{GG}, \cite{ags}, \cite{Hobson}). By Fact \ref{rankfact} our results in this section are relevant for $\sP_{2\ell}$ bundles at level one when the Lie algebra rank $\ell$ increases within a certain range. We use the results of this discussion to prove Proposition \ref{rsrank}. 

Throughout this section, we fix an integer $\ell \geq 0$ and a vector of weakly decreasing integers $\vec{\lambda} = (a_1, a_2, ..., a_n)$. We denote $\rlam$ the \textit{stabilizing Lie rank} (or the \textit{critical level} for $\sL_2$) associated to this fixed $\vec{\lambda}$ (see Definition~\ref{rlam} and Remark~\ref{CL}). 


\subsection{Ranks of bundles at varying levels for $\mathfrak{sl}_2$ and $n=4$}

Using factorization and the three pointed fusion rules from the inequalities in \ref{fusion} we show the following lemma for four pointed bundles. 
\begin{lemma}\label{4pointlemma}
For some fixed $\ell \geq 0$ and $\vec{\lambda} = (a, b, c, d)$. Let $s$ be some integer such that $\ell= \rlam+1+s$. We have the following rank relationships:
\begin{itemize}
\item[1.] If $s \leq 0$ then, (that is, $\ell$ is at or below one larger than the critical level for $\vec{\lambda}$) 
$$\rkk(\mathbb{V}(\sL_{2}, \vec{\lambda},\rlam+1+s)) = \rkk(\mathbb{V}(\sL_{2}, \vec{\lambda}, \rlam+1))+s. $$

\item[2.] If $s \geq 0$ (that is, $\ell$ is at or above one larger than the critical level for $\vec{\lambda}$), then the conformal blocks vector bundle ranks become fixed,
$$\rkk(\mathbb{V}(\sL_{2}, \vec{\lambda},\rlam+1+s)) = \rkk(\mathbb{V}(\sL_{2}, \vec{\lambda}, \rlam+1)) .$$
\end{itemize}
\end{lemma}


\begin{proof}
We compare the values of $\mu$ satisfying the inequalities in \ref{fusion} using $\ell$ and $\ell +1$. We will denote by $\mu$ a weight in the sum of Equation \ref{4pointrank} contributing a nonzero term for the rank of the bundle $\mathbb{V}(\sL_2, \vec{\lambda}, \ell)$ and $\mu'$ a weight in the sum of Equation \ref{4pointrank} contributing a nonzero term for the rank of the bundle $\mathbb{V}(\sL_2, \vec{\lambda}, \ell+1)$. With this notation, the inequalities determining such weights in \ref{fusion} can be written as follows, 
\begin{equation}\label{mu}
\mu \leq min \{ 2\ell-a-b , a+b, 2\ell-c-d, c+d\} \text{ and }
\end{equation}

\begin{equation}\label{mu'}
\mu' \leq min \{ 2\ell+2-a-b , a+b, 2\ell+2-c-d, c+d\}. 
\end{equation}

We consider the possible minimums for $\ell = \rlam +1+s$. Using $2(\rlam +1) = a+b+c+d$, we consider the cases $s \leq 0$ and $s > 0$. 

First, for $s \leq 0$, the minimums in Equations \ref{mu} and \ref{mu'} become,

\begin{equation}\label{mus0}
\mu \leq  c+d +2s \text{ and }
\end{equation}
\begin{equation}\label{mu's0}
\mu' \leq  c+d +2s + 2.
\end{equation}
We see that if $\mu$ satisfies the condition from Equation \ref{mus0} then $\mu$ also satisfies the condition for Equation \ref{mu's0}. Furthermore, if $\mu$ is the largest such value satisfying \ref{mus0} then $\mu+2$ also satisfies Equation \ref{mu's0}. Hence, the factorization sum of Equation \ref{4pointrank} for $\ell+1$ will have exactly one more nonzero term than that of Equation \ref{4pointrank} for $\ell$. This allows us to conclude $$\rkk(\mathbb{V}(\sL_{2}, \vec{\lambda}, \ell)) = \rkk(\mathbb{V}(\sL_{2}, \vec{\lambda}, \ell+1))-1. $$

Repeating this argument until $s > 0$ shows the first case.\\

Now consider when $s > 0$. Then the minimums in \ref{mu} and \ref{mu'} become,
$$\mu \leq c+d \text{ and } \mu' \leq c+d.$$
neither of which depend on the levels. This shows no nonzero terms appear in the factorization sum of Equation \ref{4pointrank} for a bundle at $\ell +1$ compared to a bundle at level $\ell$. As such, the ranks remain fixed, concluding the second case of the Lemma statement. 

\end{proof}

\subsection{Ranks of bundles at varying levels for $\mathfrak{sl}_2$ and $n \geq 4$}

We now show that the ranks of $\sL_2$ bundles with a fixed, arbitrary number of weights strictly increase or become fixed when the level is increased. 

\begin{proposition}\label{npoint}
Let $s$ be some integer such that $\ell= \rlam+1+s$. We have the following rank relationships:
\begin{itemize}
\item[1.] If $s < 0$ (that is, $\ell$ is at or below the critical level for $\vec{\lambda}$) then, 
$$\rkk(\mathbb{V}(\sL_{2}, \vec{\lambda}, \ell)) < \rkk(\mathbb{V}(\sL_{2}, \vec{\lambda}, \rlam+1)). $$

\item[2.] If $s \geq 0$ (that is, $\ell$ is strictly above the critical level for $\vec{\lambda}$), then 
$$\rkk(\mathbb{V}(\sL_{2}, \vec{\lambda}, \ell) = \rkk(\mathbb{V}(\sL_{2}, \vec{\lambda}, \rlam+1)) .$$
\end{itemize}
\end{proposition}

\begin{proof}
We argue by induction on the number of weights. 

In Lemma \ref{4pointlemma} we showed for $n=4$ the conclusion follows. For our inductive assumption, assume that for any $\sL_2$ bundle with weights $\vec{\lambda}_k= (a_1 \omega_{1}, ..., a_k \omega_{1})$ with $4 \leq k < n$ the result follows. That is, if $\rlam(k):= \sum_{i=1}^k a_i/2 -1$, the stabilizing Lie rank for $\vec{\lambda}_k$ with $k$ weights, then for $\ell = \rlam(k) +1+s $ with $s < 0$  we have a strict inequality on ranks,
$$\rkk(\mathbb{V}(\sL_{2}, \vec{\lambda}_k, \rlam(k) +1+s)) < \rkk(\mathbb{V}(\sL_{2}, \vec{\lambda}_k, \ell+1)), $$

and when $s \geq 0$ we have equality, 

$$\rkk(\mathbb{V}(\sL_{2}, \vec{\lambda}_k, \rlam(k) +1+s)) = \rkk(\mathbb{V}(\sL_{2}, \vec{\lambda}_k, \ell+1)). $$

Now suppose we have a vector of $n$ weakly decreasing weights given by $\vec{\lambda}_n= (a_1 \omega_{1}, ..., a_n \omega_{1})$. Denote the stabilizing Lie rank for $\vec{\lambda}_n$ as $\rlam(n)  = \sum_{i=1}^n a_i/2 -1$. 

For the first case, assume $\ell = \rlam +1 +s$ with $s < 0$. Using factorization (as in Equation \ref{4pointrank}) we compute the following ranks, 

\begin{equation}\label{rankl}
\rkk(\mathbb{V}(\sL_{2}, \vec{\lambda}_n,\ell)) = \sum_{\mu \in P_{\ell}^+(\sL_2)} \rkk(\mathbb{V}(\sL_{2}, (\vec{\lambda}_{n-2}, \mu), \ell))\rkk(\mathbb{V}(\sL_{2}, (a_{n-1},a_n, \mu), \ell)) \text{ and }
\end{equation}

\begin{equation}\label{rankl1}
\rkk(\mathbb{V}(\sL_{2}, \vec{\lambda}_n,\ell+1)) = \sum_{\mu \in P_{\ell+1}^+(\sL_2)} \rkk(\mathbb{V}(\sL_{2}, (\vec{\lambda}_{n-2}, \mu), \ell+1))\rkk(\mathbb{V}(\sL_{2}, (a_{n-1},a_n, \mu), \ell+1)),
\end{equation} 

where $\vec{\lambda}_{n-2}$ denotes the weight vector $\vec{\lambda}_{n-2} = (a_1, ..., a_{n-2})$. We compare each term in the sum of Equation \ref{rankl} and \ref{rankl1}. \\

First observe that if $\mu \in  P_{\ell}^+(\sL_2),$ then $\mu \in P_{\ell+1}^+(\sL_2)$ and so a weight $\mu$ that appears in the sum of Equation \ref{rankl} will also appear in Equation \ref{rankl1} (value of the ranks may be zero). With our inductive assumption, we have the following relationship between these terms,

\begin{equation}\label{comparerank}
\rkk(\mathbb{V}(\sL_{2}, (\vec{\lambda}_{n-2},\mu), \ell)) \rkk(\mathbb{V}(\sL_{2}, (a_{n-1}, a_n,\mu), \ell)) \leq \rkk(\mathbb{V}(\sL_{2}, (\vec{\lambda}_{n-2},\mu), \ell+1))\rkk(\mathbb{V}(\sL_{2}, (a_{n-1}, a_n,\mu), \ell+1)). 
\end{equation}

It follows immediately that, 

\begin{equation}\label{relate}
\rkk(\mathbb{V}(\sL_{2}, \vec{\lambda},\ell)) \leq \rkk(\mathbb{V}(\sL_{2}, \vec{\lambda},\ell+1)).
\end{equation}
We want to show that such a relationship is strict.

Consider the weight vector $(\vec{\lambda}_{n-2},\mu)$ appearing in Equation \ref{rankl}. Let $\rlam(n-2)$ denote the \textit{critical level} associated to this weight vector. From our inductive assumption, we have that the relationship in Equation \ref{comparerank} is strict whenever $\ell = \rlam(n-2) +1 +s'$ with $s' <0$. Rewriting $\rlam(n-2)$ at $(\sum_{i=1}^{n-2} a_i + \mu)/2-1$, the relationship is strict whenever $\ell \leq  (\sum_{i=1}^{n-2} a_i + \mu)/2.$ 


Hence, the inequality in \ref{relate} is strict and our conclusion follows whenever we have a $\mu$ in the sum of Equation \ref{rankl} such that $\ell \leq (\sum_{i=1}^{n-2} a_i + \mu)/2$. 

Suppose then that each $\mu$ in the sum of Equation \ref{rankl} is such that 

\begin{equation}\label{lowlevel}
\ell > (\sum_{i=1}^{n-2} a_i + \mu)/2.
\end{equation}

In this case, just comparing terms appearing in Equation \ref{rankl} and \ref{rankl1} from $\mu \in P_{\ell}(\sL_2)$ does not guarantee an increase in ranks between each term. We show that the sum in Equation \ref{rankl1} obtains an additional nonzero term not in the sum of \ref{rankl}. 

Suppose $\mu$ is the largest weight appearing as a nonzero term in Equation \ref{rankl}. Using Equation \ref{lowlevel} and recalling we are in the case that $s < 0$, we obtain,


$$\sum_{i=1}^{n-2} a_i + \mu < 2 \ell \leq 2(\ell-s) = 2(\rlam(n)+1) = \sum_{i=1}^{n} a_i .$$ 

Furthermore, this provides the two relationships,

 \begin{equation} \label{2a}
 \mu < a_{n-1} + a_n  \\
 \end{equation}
 \begin{equation}\label{2b}
a_{n-1} + a_n + \mu < 2 \ell.
\end{equation}

Relating these inequalities, we obtain a strict inequality, $\mu < \ell$ from which the weak relationship follows, $\mu+2 \leq \ell +1$. Hence, that weight $\mu+2$ appears as a possible weight for a term in the sum of Equation \ref{rankl1}. Furthermore, comparing Equation \ref{2a} and \ref{2b} with the three point fusion rules (i.e., the inequalities in Equation \ref{fusion}), we can conclude the rank term for this weight $\rkk(\mathbb{V}(\sL_{2}, (a_{n-1}, a_n,\mu), \ell+1))$ is nonzero in the sum of Equation \ref{rankl}. We finally just need to analyze the rank of the other factor in this term.

By Equation \ref{lowlevel}, for our fixed largest $\mu$ in the sum of Equation \ref{rankl}, we must have that $\sum_{i=1}^{n-2} a_i + \mu =2p$ for some $p < \ell$ and so also $\sum_{i=1}^{n-2} a_i + \mu +2 =2(p+1)$. Now, since we have $\rkk(\mathbb{V}(\sL_{2}, (a_1, ..., a_{n-2},\mu), \ell)) > 0$ (we have assumed $\mu$ appears as a weight with a nonzero rank term for level $\ell$), it follows from the condition for nonzero rank of $\sL_2$ bundles in \cite[Thm.~1.1]{Hobson} that $\sum_{i=2}^{n-2} a_i + \mu \geq p$ (that is, the sum of the last $n-2$ weights is greater than or equal to $p$). From this relationship it follows that, $\sum_{i=2}^{n-2} a_i + \mu +2 \geq p +1.$ Again, using the same result \cite[Thm.~1.1]{Hobson}, we can conclude $\rkk(\mathbb{V}(\sL_{2}, (a_{n-1}, a_n,\mu), \ell+1)) > 0$.


Since $\mu$ was assumed to be the largest weight in Equation \ref{rankl}, the weight $\mu+2$ does not appear in the rank calculation for level $\ell$ but does contributes a new nonzero term in Equation \ref{rankl1}. This allows us to conclude the relationship in Equation \ref{relate} is strict.\\

In the second case, with $\ell = \rlam +1 +s$  and $s \geq 0$ the rank computation using Witten's Dictionary \cite{BelkaleWittenDic} is the same calculation for all $s$ in this range, this shows equality of ranks. This was used in the proof of vanishing above critical level \cite[Section~4]{BGMA}. 
\end{proof}

In Section \ref{examples} we provide an example of this rank behavior with a bundle of type C at level one (see Example \ref{stableExample}). 

\begin{remark}\label{previous}
There has been previous work by Alex Yong and others on results related to bounds on structure constants for the product of Schubert classes (\cite{APostnikov, ABuch, AYong}). However, such results compare values within a fixed quantum cohomology ring of the Grassmannian (see \cite{BelkaleWittenDic}). The result of Proposition~\ref{npoint} shows behavior of structure constants appearing in products of Schubert classes across rings (that is, the parameters of the Grassmannian vary).
\end{remark}




\section{Proof of Stabilizing Lie Rank}\label{stabilizingproof}

To prove Proposition \ref{rsdivisor} we show the bundles of interest have equal intersection with any $F$-curve. To make this comparison, we first establish the result for bundles on $\M_{0,4}$. We use notation to match the formulas in Lemmas~\ref{DegreeSP}.

%



\begin{lemma}\label{4pointdegree}
For a fixed level $\ell$ and $\vec{\lambda}=(a,b,c,d)$, let $\rlam$ be the stabilizing Lie rank and let $s$ be some integer such that $\ell+s= \rlam+1$. If $s \leq 1$ then we have equality, $$\deg(\mathbb{V}(\sP_{2\rlam}, \vec{\lambda}, 1)) = \deg(\mathbb{V}(\sP_{2\ell},\vec{\lambda}, 1)).$$
\end{lemma}

This result follows immediately from comparing the formulas in Lemmas \ref{DegreeSP}. Using the language of the stabilizing Lie rank, this Lemma says that the degrees of divisors with $\sP_{2\ell}$ at level one with four weights become equal when the Lie rank $\ell$ is chosen to be at or above the stabilizing Lie rank for the weight vector, $\vec{\lambda}$. We are now ready to prove this result for an arbitrary number of weights.


\begin{proof}[Proof of Proposition \ref{rsdivisor}]
Let $\vec{\lambda}=(a_1, ..., a_n)$ be an $n$-tuple of weakly decreasing integers such that $|\vec{\lambda}|=\sum_{i=1}^n (a_i)$ is even. Using Definition \ref{def1}, the stabilizing Lie rank is the integer $\rlam$ such that $|\vec{\lambda}|=2(\rlam +1)$. Now suppose $\ell$ is some integer such that $\ell \geq a_1$ and $\ell > \rlam$. Thus, we can write $|\vec{\lambda}| = 2(\ell+s)$ with $s \leq 0$. We want to show equivalence of the divisors, 
$$c_1(\rlam) \text{ and } c_1(\VlL).$$

We compare the intersection numbers of these two bundles with an arbitrary $F$-curve determined by a partition $\{1, ..., n\} = I_1 \sqcup I_2 \sqcup I_3 \sqcup I_4$. The formula in \cite[Prop.~2.7]{Fakhruddin} provides the following computation. We denote $\mathbb{V}_{\vec{\lambda}, \ell}:= \mathbb{V}(\sP_{\ell}, \vec{\lambda}, 1)$ to simply notation throughout this section.

\begin{equation}\label{FacDegreerlam}
\deg(\mathbb{V}_{\vec{\lambda}, \rlam}| F_{I_1, I_2,I_3,I_4}) = \sum_{\vec{\mu}\in P(\mathfrak{g}, \rlam)^4} \deg(\mathbb{V}_{\vec{\mu}, \rlam})\rkk(\mathbb{V}_{(\lambda_{I_1}, \mu_1), \rlam})\rkk(\mathbb{V}_{(\lambda_{I_2}, \mu_2), \rlam})\rkk(\mathbb{V}_{(\lambda_{I_3}, \mu_3), \rlam})\rkk(\mathbb{V}_{(\lambda_{I_4}, \mu_4), \rlam})
\end{equation}

\begin{equation}\label{FacDegreer}
\deg(\mathbb{V}_{\vec{\lambda}, \ell}| F_{I_1, I_2,I_3,I_4}) = \sum_{\vec{\nu}\in P(\mathfrak{g}, \ell)^4} \deg(\mathbb{V}_{\vec{\nu}, \ell})\rkk(\mathbb{V}_{(\lambda_{I_1}, \nu_1), \ell}) \rkk(\mathbb{V}_{(\lambda_{I_2}, \nu_2), \ell})\rkk(\mathbb{V}_{(\lambda_{I_3}, \nu_3), \ell})\rkk(\mathbb{V}_{(\lambda_{I_4}, \nu_4), \ell})
\end{equation}

where $\lambda_{I_j}$ denotes the weight vector with weights $a_i$ for $i \in I_j$ and $\vec{\mu} = (\mu_1, \mu_2, \mu_3, \mu_4)$ and $\vec{\nu} = (\nu_1, \nu_2, \nu_3, \nu_4)$ denote the attaching weight vectors. 


First, we show that the attaching weights  $\vec{\nu} = (\nu_1, \nu_2, \nu_3, \nu_4)$ appearing in the degree term of Equation \ref{FacDegreer} are all such that $\nu_i \leq \rlam$. This will allow us to take the above sums over the same set of integers.  

Consider all possible terms in Equation \ref{FacDegreer} and let $|I_i| := \sum_{j \in I_i} a_j$, the sum of just those weights appearing in a partition determined by $I_i$. By the the Generalized Triangle Inequality for ranks of $\sL_2$ bundles \cite[Lemma~3.8]{ags} and Fact \ref{rankfact}, in order for the rank of $\mathbb{V}_{(\lambda_{I_i}, \nu_i), \ell}$ to be nonzero it is necessary that 
\begin{equation*}\label{gentriangle} 
\nu_i \leq |I_i|.
\end{equation*}
In order for a term in Equation \ref{FacDegreer} to be nonzero, it is necessary that this condition holds for $i=1, 2, 3, 4$ . Adding all such inequalities gives, 
\begin{equation}\label{genineq}
\sum_{i=1}^4 \nu_i \leq \sum_{i=1}^4 |I_i| = \sum_{i=1}^na_i= 2(\rlam + 1).
\end{equation}

 Suppose for some weight $\nu_k \in \vec{\nu}$, we had $\rlam + 1 < \nu_k$; it would follow that $$\sum_{i=1}^4 |\nu_i| \leq \sum_{i=1}^4 |I_i| =  2(\rlam + 1) < 2(\nu_k).$$ Canceling $\nu_k$ from this inequality would imply that $\rk(\mathbb{V}_{\vec{\nu}, \ell}) = 0$ (from applying the Generalized Triangle Inequality with these weights) and so the degree, $\deg(\mathbb{V}_{\vec{\nu}, \ell})$, is zero. This shows that all nonzero terms in Equation \ref{FacDegreer} have attaching data $\vec{\nu}$ such that each $\nu_i \leq \rlam + 1 $. We need to check that in fact this inequality is strict so that all attaching data in Equation \ref{FacDegreer} are $\leq \rlam$. To accomplish this, we compute $\deg(\mathbb{V}_{\vec{\nu}, \ell})$ and assume for some $k$, $\nu_k = \rlam +1$. Using Equation \ref{DegreeSP}, we obtain $\deg(\mathbb{V}_{\vec{\nu}, \ell})=0$ in this case as well.   
 


We now have that all nonzero terms appearing in the sums of Equation \ref{FacDegreerlam} and \ref{FacDegreer} have attaching weights $\vec{\mu}$ and $\vec{\nu}$ with $\mu_i, \nu_i \leq \rlam$. Particularly, each term in these sums have the same attaching data. To finish the proof, we compare corresponding terms in each sum and show they are equal.

First, we compare rank terms. Consider the following ranks, 

$$\rkk(\mathbb{V}_{(\lambda_{I_i}, \mu_i), \rlam}) \text{ and } \rkk(\mathbb{V}_{(\lambda_{I_i}, \mu_i), \ell}).$$

Define $\sigma_i$ to be the integer such that $2\sigma_i := |\lambda_{I_i}|+\mu_i$, then from Lemma \ref{npoint}, we see that these ranks will be equal whenever we have $\sigma_i \leq \rlam$. So suppose $\rlam < \sigma_k$ for some $k = 1, 2, 3,$ or $4$. Without loss of generality (and clarity in the following argument) we assume $k=1$. From this, it follows $$\sum_{i=1}^4 |\lambda_{I_i}| = 2(\rlam + 1) \leq 2\sigma_1 = |\lambda_{I_1}|+\mu_1,$$ (where the strict inequality of our assumption provides the weak inequality $\rlam +1 \leq \sigma_1$). Canceling $|I_1|$ gives the relationship, $\sum_{i=2}^4 |\lambda_{I_i}| \leq \mu_1.$

Using this relationship and applying Equation \ref{genineq} and the Generalized Triangle Inequality for $\rkk(\mathbb{V}_{\vec{\mu}, \ell})$ we have, $$\mu_2+\mu_3+\mu_4 \leq \sum_{i=2}^4 |\lambda_{I_i}| \leq \mu_1 \leq \mu_2+\mu_3+\mu_4.$$ Obtaining the equality, $$\mu_1 = \mu_2+\mu_3+\mu_4.$$


Consider the degree terms $\deg(\mathbb{V}_{\vec{\mu}, \rlam})$ and $\deg(\mathbb{V}_{\vec{\mu}, \ell})$ appearing in Equations \ref{FacDegreerlam} and \ref{FacDegreer} for the attaching weight with $\mu_1=\mu_2+\mu_3+\mu_4$. By Lemma \ref{DegreeSP} (following the first formula for either degrees) 
the degree is zero. Hence, terms in the sums with such $\lambda_1, \mu_1$ are also zero. Thus, we can always assume $\sigma_i \leq \rlam$ and thus rank terms are equal.

We now compare the corresponding degree terms in each sum of Equation \ref{FacDegreerlam} and \ref{FacDegreer}.  In order for the product of rank terms (in either sum) to not necessarily be zero, we have the relationship in Equation \ref{genineq} (now using notation $\vec{\mu}= (\mu_1, \mu_2, \mu_3, \mu_4)$). From this relationship, it follows from Lemma \ref{4pointdegree} that the four pointed degree terms are equal.

We can now conclude that the terms appearing in \ref{FacDegreerlam} and \ref{FacDegreer} are always equal and so we conclude the proposition statement.
\end{proof}

It was shown for $\sL_2$, that for $\mathbb{V}=\mathbb{V}(\sL_2, \vec{\lambda}, \ell)$ to be nontrivial is equivalent to $\ell \leq \rlam$ and $0 < \rkk(\mathbb{V})$ (see \cite{BGMA}). Considering the degree formula for $\sP_{2\ell}$ divisors in Lemmas \ref{DegreeSP}, the nontrivality of $\sP_{2r}$ divisors above stabilizing Lie rank follows (if rank the stabilizing Lie bundle is nonzero). This should be compared to the vanishing result in \cite[Cor.~3.6]{BGMB}. See Corollary \ref{non} for a further nonvanishing statement for bundles of $\sP_{2r}$ at level one.

\begin{corollary}\label{nonVanishing}
Let $c_1(\mathbb{V}(\sP_{2\rlam}, \vec{\lambda}, 1))$ be the  stable Lie divisor for a fixed $n$-tuple, $\vec{\lambda}$. Then if $\rk(\mathbb{V}(\sP_{2\rlam}, \vec{\lambda}, 1)) >0$, the divisor $c_1(\mathbb{V}(\sP_{2r}, \vec{\lambda}, 1))$ is nontrivial for all $r \geq \rlam$.
\end{corollary}

\section{Examples}\label{examples}

Here we give examples of some of the results we have shown. In the first two examples we provide bundles $\VsL$ and $\VsP$ such that $c_1(\VsL) \neq c_1(\VsP)$ such that $\VsL$ has certain special properties described in previous work. Particularly, in Example \ref{rank2} we provide an $\sL_2$ bundle that has \textit{projective rank scaling} (\cite[Def.~2.8]{BGK}) and in Example \ref{aboveCL} the level (or Lie rank) is above the critical level for $\sL_2$ (see Remark \ref{CL} and \cite[Def.~1.1]{BGMA}). In Example \ref{stableExample} we given an example of type C bundle at level one to demonstrate the rank behavior of Lemma \ref{npoint} and stable Lie divisor of Proposition \ref{rsdivisor} for type C bundles. 

\begin{example}\label{rank2}
Let $\ell = 5$ and weights be given by $\vec{\lambda}=(4,4,4,4)$. Consider the bundles:

$$\VsL = \mathbb{V}(\sL_2, (4,4,4,4), 5) \text{ and } \VsP =  \mathbb{V}(\sP_{2 \cdot 5}, (4,4,4,4), 1).$$

We have that $|\vec{\lambda}| = 16= 2( 5 + 1)$ (showing that $\ell=5$ is at the critical level for $\VsL$). The ranks of these bundles are, $\rk(\VsL)=\rk(\VsP)=2$. However, comparing the degrees give $\deg(\VsL) =6$ while $ \deg(\VsP)=7$.
\end{example}

Note that in Example \ref{rank2} since $\rkk(\VsL) =2$, the bundle $\VsL$ is said to have \textit {projective rank scaling} (\cite[Def.~2.8]{BGK}).

\begin{example}\label{aboveCL}
Let $\ell = 5$ and $\vec{\lambda} = (2,2,1,1)$. Consider the bundles:

$$\VsL = \mathbb{V}(\sL_2, (2, 2, 1, 1), 5) \text{ and } \VsP =  \mathbb{V}(\sP_{2\cdot5}, (2, 2, 1, 1), 1).$$

Computing ranks gives, $\rk(\VsL)=\rk(\VsP)=2.$ Furthermore, we have that $|\vec{\lambda}| = 6= 2( 5 - 2)$ and so $\ell$ is above the critical level (or stabilizing Lie rank). Because of this, the bundle $\VsL$ is trivial and so $\deg(\VsL)=0$ (see \cite{BGK}), however for $\VsP$ we find, $\deg(\VsP)=1$.
\end{example}

\begin{example}\label{n6AllCoef}
Let $\ell =5$ and $\vec{\lambda} = (4, 4, 4, 4, 3, 3)$. Consider the bundles,
$$\VsL = \mathbb{V}(\sL_2, (4,4,4,4,3,3), 5) \text{ and } \VsP =  \mathbb{V}(\sP_{2\cdot5}, 4,4,4,4,3,3), 1).$$ 

Computing ranks, we have $\rkk(\VsL)=\rkk(\VsP) =2$. Using the program ``Special6.m2'' written by D. Krashen, we can explicitly write each divisor in this basis in the nonadjacent basis of Pic($\M_{0,6}$) (see \cite[Example~4.4]{MoonSwin}). We give the coordinates of each divisor with the basis ordered as $\{ \delta_{13}, \delta_{14}, \delta_{15}, \delta_{24}, \delta_{25}, \delta_{26}, \delta_{35}, \delta_{36}, \delta_{46}, \delta_{124}, \delta_{125}, \delta_{134}, \delta_{135}, \delta_{136}, \delta_{145}, \delta_{146} \}$. 

This computation gives,
$$c_1(\VsL) = (12, 6, 12, 12, 6, 12, 12, 0, 12, 2, 2, 6, 24, 2, 2, 6) \text{ and }$$
$$c_1(\VsP) = (14, 8, 14, 14, 8, 14, 14, 3, 14, 4, 4, 8, 28, 4, 4, 8).$$
\end{example}

\begin{example}\label{stableExample}
Let $\vec{\lambda} = (5,4,3,2,1,1)$. We consider the type C bundles, $\mathbb{V}_{\ell} = \mathbb{V}(\sP_{2\ell}, \vec{\lambda}, 1)$ for varying $\ell$. Using the same ordering on the nonadjacent basis of Pic($\M_{0,6}$) as in Example \ref{aboveCL}, we compute the divisor class and rank of each $\mathbb{V}_{\ell}$. The stabilizing Lie rank for $\vec{\lambda}$ is $\rlam = 7$. By Proposition \ref{rsdivisor} the divisors defined at or above $\rlam=7$ are all equal. By Lemma \ref{npoint} the rank of the bundle with Lie algebra rank at or above $\rlam+1$ are equal. This is shown from computations displayed in Table \ref{stableExampleTab}.

\begin{table}
\centering
\begin{tabular}{|l|l|l|}
\hline
$\ell$ & $c_1(\mathbb{V}_{\ell})$                                & $\rkk(\mathbb{V}_{\ell})$ \\ \hline
5      & $(7, 1, 1, 5, 2, 2, 1, 1, 1, 1, 1, 3, 7, 6, 1, 1)$      & 3                         \\ \hline
6      & $(11, 4, 2, 9, 4, 4, 3, 3, 2, 4, 2, 6, 12, 10, 3, 3)$   & 7                         \\ \hline
7      & $( 12, 5, 3, 10, 5, 5, 4, 4, 3, 5, 3, 7, 14, 11, 4, 4)$ & 10                        \\ \hline
8      & $( 12, 5, 3, 10, 5, 5, 4, 4, 3, 5, 3, 7, 14, 11, 4, 4)$ & 11                        \\ \hline
9      & $( 12, 5, 3, 10, 5, 5, 4, 4, 3, 5, 3, 7, 14, 11, 4, 4)$ & 11                        \\ \hline
10     & $( 12, 5, 3, 10, 5, 5, 4, 4, 3, 5, 3, 7, 14, 11, 4, 4)$ & 11                        \\ \hline
\end{tabular}
\caption{Divisors and ranks for $\mathbb{V}_{\ell} = \mathbb{V}(\sP_{2\ell}, \vec{\lambda}, 1)$ and varying $\ell$}
\label{stableExampleTab}
\end{table}

\end{example}

\section{Consequences for conformal blocks of Type C at level one }\label{corollaries}

The main propositions of this paper have several consequences to the study of understanding conformal blocks divisors in nef$(\M_{0,n})$. Specifically, we combine the results of Proposition \ref{main} with previous results related to vector bundles of conformal blocks with $\sL_2$ to conclude several consequences for bundles with $\sP_{2\ell}$ at level one.

The finite generation of the cone of all conformal blocks divisors in nef$(\M_{0,n})$ is an open question. This problem was considered for conformal blocks with $\sL_n$ at level one in \cite{GG} and for bundles with $\sL_2$ of rank one in \cite{Hobson}. Using these results, we are able to make the following conclusion related to conformal blocks with $\sP_{2\ell}$ at level one.

\begin{corollary}\label{divisors}
For $\mathbb{V}= \mathbb{V}(\mathfrak{sp}_{2\ell}, 1, \vec{\lambda})$. Let $$\mathcal{C}=convHull \{ c_1(\mathbb{V}) : \rkk(\mathbb{V}) = 1 \}.$$

 The cone $\mathcal{C}$ is the same as that generated by conformal blocks divisors with $\mathfrak{sl}_2$ and rank one. Particularly, such a cone is finitely generated.
\end{corollary}

Another open problem in this study is to determine necessary and sufficient conditions for when a conformal blocks divisor is nonzero (\cite[Question~0.1]{BGMB}). Due to results in \cite[Prop.~4.3]{Fakhruddin} and \cite[Cor.~3.6]{BGMB} we can conclude the following nonvanishing result for $\sP_{2\ell}$ conformal blocks at level one. 

\begin{corollary}\label{non}
For $\VsP$ as defined in Def. \ref{def1} with some fixed integer $\ell$ and $n$-tupe $\vec{\lambda}$, we have the following nonvanishing result $$c_1(\VsP) \text{ is nontrivial } \Leftrightarrow \rkk(\VsP) > 0 \text{ and } \rkk(\Vrlam) > 0.$$ 
\end{corollary}


Additionally, using the decomposition of \cite[Prop.~1.2]{BGMB}, Proposition \ref{main} provides new decomposition and scaling identities for the divisors $c_1(\VsP)$ with $\rkk(\VsP) =1$.
\begin{corollary}
For $\vec{\mu}=( \omega_{a_1}, ..., \omega_{a_n})$ such that $0 < a_i \leq m$ and $\vec{\nu}=( \omega_{b_1}, ..., \omega_{b_n})$ such that $0 < b_i\leq \ell$ (so that $\vec{\mu} \in P(\sP_{2m},1)^n$ and $\vec{\nu} \in P(\sP_{2\ell}, 1)^n$ are dominant integral weights for $\sP_{2m}$ and $\sP_{2\ell}$ at level one respectively). If $\rk( \mathbb{V}(\mathfrak{sp}_{2m}, 1, \vec{\mu}) )=  \rk(\mathbb{V}(\mathfrak{sp}_{2\ell}, 1, \vec{\nu})) =1$ then 

$$ c_1(\mathbb{V}(\mathfrak{sp}_{2(m+\ell)}, 1, (\omega_{a_1+b_1}, ..., \omega_{a_n+b_n}))) = c_1(\mathbb{V}(\mathfrak{sp}_{2m}, 1, \vec{\mu})) +  c_1(\mathbb{V}(\mathfrak{sp}_{2\ell}, 1, \vec{\nu})).$$
\end{corollary}

Iterating this result leads to the following scaling behavior. 
 \begin{corollary}\label{CorScale}
Define $\mathbb{V}_N:=\mathbb{V}(\mathfrak{sp}_{2(N\ell)}, 1, ( \omega_{Na_1}, ..., \omega_{Na_n}))$ for $( \omega_{a_1}, ..., \omega_{a_n}) \in  P_1(\sP_{2\ell}, 1)^n$. If $\mathbb{V}_1$ has rank one, then we have the following divisor identity:
$$c_1(\mathbb{V}_N) = N c_1(\mathbb{V}_1).$$
\end{corollary}

\begin{remark}
Similar scaling behavior appears for $\mathfrak{sl}_{r+1}$ in \cite[Cor.~4.6]{BGMB}, for $\mathfrak{sl}_2$ and $\vec{\lambda}=(\omega_1, ..., \omega_1)$ in \cite[Prop.~5.2]{GJMS13}, and analogously for $\mathfrak{so}_{2r+1}$ and $\vec{\lambda}=(\omega_1, ..., \omega_1)$ in \cite[Thm.~1.2]{MUK14}.
\end{remark}

\bibliography{TypeCLevelOneArXiv}

\end{document}